\documentclass[12pt,a4paper,oneside]{article}
\usepackage[utf8]{inputenc}
\usepackage{amsmath}
\usepackage{amsfonts}
\usepackage{amssymb}
\usepackage{graphicx}
\usepackage[english]{babel}
\usepackage{color}
\usepackage{tikz}
\usepackage{fancyhdr}
\usepackage[T1]{fontenc}
\usepackage{nicefrac}
\usepackage{dsfont}
\usepackage{amsthm}
\usepackage{soul}
\usepackage[colorlinks=true,linkcolor=blue]{hyperref}
\makeatletter
\renewcommand*{\eqref}[1]{%
  \hyperref[{#1}]{\textup{\tagform@{\ref*{#1}}}}%
}

\usetikzlibrary{calc,backgrounds,fit}
\usetikzlibrary{external}
\usetikzlibrary{patterns}
\usetikzlibrary{decorations.pathreplacing,decorations.markings,arrows}
\usetikzlibrary{shapes,intersections,positioning}
\usetikzlibrary{arrows}

\definecolor{kerstin}{RGB}{0,150,75}

\newcommand\centre[4][below]{\node (#3) at #2 [circle,minimum size=0.5em,inner sep=0pt,thin,fill,solid] {}; \node [#1=0.1em] at (#3) {#4};}

\newcounter{cst}
\newcommand{\ctel}[1]{K_{\refstepcounter{cst}\label{#1}\thecst}}
\newcommand{\cter}[1]{K_{\ref*{#1}}} 

\newcommand{\re}{\mathbb{R}}
\newcommand{\na}{\mathbb{N}}

\newcommand{\eps}{\varepsilon}

\newcommand{\sign}{\operatorname{sign}}

\newcommand{\lz}{L^2(\re^d)}

\newcommand{\erwb}{\mathbb{E}\left[}
\newcommand{\erwe}{\right]}

\newcommand{\halbe}{\frac{1}{2}}

\newcommand{\Tau}{\mathcal{T}}
\newcommand{\edges}{\mathcal{E}}
\newcommand{\dkl}{d_{K|L}}
\newcommand{\edgesint}{\mathcal{E}_{\operatorname{int}}}
\newcommand{\edgesext}{\mathcal{E}_{\operatorname{ext}}}

\newcommand{\uhnl}{u_{h,N}^l}
\newcommand{\uhnr}{u_{h,N}^r}
\newcommand{\uhnlm}{u_{h_m,N_m}^l}
\newcommand{\uhnrm}{u_{h_m,N_m}^r}

\newcommand{\net}{{n_1(t)}}
\newcommand{\nnt}{{n_0(t)}}
\newcommand{\lzlambda}{{L^2(\Lambda)}}

\newcommand{\buhnl}{\bar{u}_{h,N}^l}

\newcommand{\vhnl}{v_{h,N}^l}
\newcommand{\vhnr}{v_{h,N}^r}
\newcommand{\vhnlm}{v_{h_m,N_m}^l}

\newcommand{\whnl}{w_{h,N}^l}
\newcommand{\whnr}{w_{h,N}^r}
\newcommand{\erws}[1]{\mathbb{E}'\left[#1\right]}
\newcommand{\erww}[1]{\mathbb{E}\left[{#1}\right]}
\newcommand{\reg}{\operatorname{reg}(\Tau)}

\newtheorem{defi}{Definition}[section]

\newtheorem{lem}[defi]{Lemma}
\newtheorem{teo}[defi]{Theorem}
\newtheorem{prop}[defi]{Proposition}

\theoremstyle{remark}
\newtheorem{remark}[defi]{Remark}
\newtheorem*{notation}{Notation}

\numberwithin{equation}{section}

\usepackage[normalem]{ulem} 
\usepackage{cancel}

\title{Convergence of a finite-volume scheme for a heat equation with a multiplicative Lipschitz noise}
\author{Caroline Bauzet\footnotemark[1], \and Flore Nabet\footnotemark[2], \and Kerstin Schmitz\footnotemark[3], \and Aleksandra Zimmermann\footnotemark[3]}
\date{}

\begin{document}
\maketitle

\begin{abstract} We study here the approximation by a finite-volume scheme of a heat equation forced by a Lipschitz continuous multiplicative noise in the sense of It\^o. More precisely, we consider a discretization which is semi-implicit in time and a two-point flux approximation scheme (TPFA) in space. 
We adapt the method based on the theorem of Prokhorov to obtain a convergence in distribution result, then Skorokhod's representation theorem yields the convergence of the scheme towards a martingale solution and the Gy\"{o}ngy-Krylov argument is used to prove convergence in probability of the scheme towards the unique variational solution of our parabolic problem.
\end{abstract}
\footnotetext[1]{Aix Marseille Univ, CNRS, Centrale Marseille, LMA UMR 7031, Marseille, France, caroline.bauzet@univ-amu.fr, $+$33 4 84 52 56 31 }
\footnotetext[2]{CMAP, CNRS, Ecole Polytechnique, Institut Polytechnique de Paris, 91128 Palaiseau, France, flore.nabet@polytechnique.edu, $+$33 1 69 33 45 65}
\footnotetext[3]{Universit\"at Duisburg-Essen, Fakult\"at f\"ur Mathematik, Essen, Germany, kerstin.schmitz@uni-due.de, $+$49 201 183 6130, aleksandra.zimmermann@uni-due.de, $+$49 201 183 7292}

\section{Introduction}

Let $\Lambda\subset\re^2$ be a bounded, open, connected and  polygonal set.
Moreover let $(\Omega,\mathcal{A},\mathds{P})$ be a probability space endowed with a right-continuous, complete filtration $(\mathcal{F}_t)_{t\geq 0}$ and let $(W(t))_{t\geq 0}$ be a standard one-dimensional Brownian motion with respect to $(\mathcal{F}_t)_{t\geq 0}$ on $(\Omega,\mathcal{A},\mathds{P})$.\\
For $T>0$, we consider a nonlinear stochastic heat equation under Neumann boundary conditions:
\begin{align}\label{equation}
\begin{aligned}
du-\Delta u\,dt &=g(u)\,dW(t), &&\text{ in }\Omega\times(0,T)\times\Lambda;\\
u(0,\cdot)&=u_0, &&\text{ in } \Omega\times\Lambda;\\
\nabla u\cdot \mathbf{n}&=0, &&\text{ on }\Omega\times(0,T)\times\partial\Lambda;
\end{aligned}
\end{align}
where $\mathbf{n}$ denotes the unit normal vector to $\partial\Lambda$ outward to $\Lambda$.
We assume the following hypotheses on the data:
\begin{itemize}
\item[$H_1$:] $u_0\in L^2(\Omega;H^1(\Lambda))$ is $\mathcal{F}_0$-measurable.
\item[$H_2$:] $g:\re\rightarrow\re$ is a Lipschitz continuous function with Lipschitz constant $L\geq 0$.
\end{itemize}
\begin{remark}
$H_2$ implies 
\begin{equation}\label{H3}
|g(r)|^2\leq C_L(1+|r|^2)
\end{equation}
for all $r\in\mathbb{R}$ and a constant $C_L\geq 0$ only depending on the Lipschitz constant $L\geq 0$ of $g$ and on $g(0)$. In particular, our scheme applies for square integrable, additive noise with appropriate measurability assumptions.
\end{remark}

\subsection{Concept of solution and main result}
The theoretical framework associated with Problem~\eqref{equation} is well established in the literature. Indeed, we can find many existence and uniqueness results for various concepts of solutions associated with this problem such as mild solutions, variational solutions, pathwise solutions and weak solutions, see e.g. \cite{DPZ14} and \cite{LR}. In the present paper we will be interested in the concept of solution as defined below, which we will call a variational solution: 
\begin{defi}\label{solution}
A variational solution to Problem \eqref{equation} is an $(\mathcal{F}_t)_{t\geq0}$-adapted stochastic process
\begin{align*}
u\in L^2(\Omega;C([0,T];L^2(\Lambda)))\cap L^2(\Omega;L^2(0,T;H^1(\Lambda)))
\end{align*}
such that for all $t\in[0,T]$,
\begin{align*}
u(t)-u_0-\int_0^t \Delta u(s)\,ds=\int_0^t g(u(s))\,dW(s)
\end{align*}
in $L^2(\Lambda)$ and  a.s. in $\Omega$.
\end{defi}
Existence, uniqueness and regularity of this variational solution is well-known in the literature, see e.g. \cite{PardouxThese},\cite{KryRoz81},\cite{LR}.
The main result of this paper is to propose a finite-volume scheme for the approximation of such a variational solution and to show its stochastically strong convergence by passing to the limit with respect to the time and space discretization parameters.  This is stated in the following convergence result: 
\begin{teo} \label{mainresult}
Assume that hypotheses $H_1$ and $H_2$ hold. 
Let $(\Tau_m)_{m\in \mathbb{N}}$ be a sequence of admissible finite-volume meshes of  $\Lambda$ in the sense of Definition \ref{defmesh} such that the mesh size $h_m$ tends to $0$ and $(N_m)_{m\in \mathbb{N}}\subseteq \mathbb{N}^*$ a sequence of positive numbers which tends to infinity.
For a fixed $m\in\mathbb{N}$, let $u^r_{h_m,N_m}$ and $u^l_{h_m,N_m}$ be respectively the right and left in time finite-volume approximations defined by \eqref{eq:notation_wh}, \eqref{eq:def_u0}-\eqref{equationapprox} with $\Tau =\Tau_m$ and $N=N_m$. Then $(u^r_{h_m,N_m})_{m\in \mathbb{N}}$ and $(u^l_{h_m,N_m})_{m\in \mathbb{N}}$ converge in $L^p(\Omega;L^2(0,T;L^2(\Lambda)))$ for any $p\in[1,2)$ to the variational solution of Problem \eqref{equation} in the sense of Definition~\ref{solution}.
\end{teo}

\subsection{State of the art}
The study of numerical schemes for stochastic partial differential equations (SPDEs) has attracted a lot of attention in the last decades and there exists an extensive literature on this topic.
A list of references for the numerical analysis of SPDEs and an overview of the state of the art is given in \cite{DP09}, \cite{ACQS20} and \cite{OPW20}.\\ 
Concerning the theoretical and numerical study of stochastic heat equations, semigroup techniques may be used to construct mild solutions (see, e.g., \cite{DPZ14}). However, from the point of view of applications and mathematical modeling, it is often interesting to consider first-order perturbations of the stochastic heat equation and more complicated, nonlinear second order operators, such as the $p$-Laplacian or the porous medium operator. For these nonlinear SPDEs, the semigroup approach is not available and variational techniques have been developed in \cite{PardouxThese}, \cite{KryRoz81} and \cite{LR}.\\
In the numerical analysis of variational solutions to parabolic SPDEs, spatial discretizations of finite-element type have been frequently used (see, e.g., \cite{Prohl}, \cite{BHL21} and the references therein). On the other hand, for stochastic scalar conservation laws, finite-volume schemes have been studied in \cite{BCG161}, \cite{BCG162},  \cite{BCG17}, \cite{FGH18} , \cite{M18}, \cite{DV18}, \cite{BCC20} and \cite{DV20}.
To the best of our knowledge, there are only a few results on finite-volume schemes for parabolic SPDEs. Let us mention the work of  \cite{BN20} where the authors proposed a convergence result of a finite-volume scheme for the approximation of  a stochastic heat equation with linear multiplicative noise.

\subsection{Aim of the study}
In this contribution, we want to extend the finite-volume approximation results in the hyperbolic case to the stochastic heat equation with Lipschitz continuous multiplicative noise. Having applications to nonlinear operators and also to degenerate parabolic-hyperbolic problems with stochastic force in mind for the future, we propose a method for the convergence of the scheme which does not rely on mild solutions and results from semigroup theory.
Additionally, we may include a discrete gradient in the right-hand side of our scheme \eqref{equationapprox} in the future.  Hence, further studies may be devoted to the convergence analysis of finite-volume schemes for equations with multiplicative noise involving first order spatial derivatives of the solution.
The main technical challenge is the nonlinear multiplicative noise. Indeed, from the \textit{a priori} estimates, we get up to subsequences weak convergence results in several functional spaces for  our finite-volume approximations and this mode of convergence is not enough to identify the weak limit of the nonlinear term in the stochastic integral.
Therefore, we first show the convergence towards a martingale solution by adapting the stochastic compactness method based on Skorokhod's representation theorem. Then, using a famous argument of pathwise uniqueness (see, e.g., \cite{Krylov}), we obtain the stochastically strong convergence result stated in Theorem \ref{mainresult}.
In this contribution, we limit ourselfes to the convergence proof. The study of convergence rates is subject to current research activities and will be detailed in a forthcoming work.

\subsection{Outline}
The paper is organized as follows. The next section concerns the introduction of the finite-volume framework: the definition of an admissible finite-volume mesh on $\Lambda$ will be stated and the associated notations of discrete unknowns will be given. Then the notions of discrete gradient and discrete $H^1$-seminorm will be introduced. In a last subsection, we will introduce our finite-volume scheme together with the associated finite-volume approximations. The remainder of the paper is then devoted to the proof of the convergence of this approximations towards the variational solution of \eqref{equation}. To do so, we will prove in Section \ref{estimates} several stability estimates satisfied by these approximations, but also a boundedness result on the approximation of the stochastic integral. These estimates will allow us to pass the limit in the numerical scheme in Section \ref{ConvFVscheme}. More precisely, we apply the classical stochastic compactness argument (see, e.g., \cite{BreitFeireisl}). By the theorem of Prokhorov, we will get convergence in law (up to subsequences) of our finite-volume approximations. At the cost of a change of probability space, the  Skorokhod representation theorem will allow us to obtain almost sure convergence of the proposed finite-volume scheme. Then, a martingale identification argument will help us in order to recover at the limit the desired stochastic integral.
In this way, we have shown that our finite-volume scheme converges to a martingale solution of \eqref{equation}, i.e., the stochastic basis is not fixed, but enters an unknown in the equation. Next, we show pathwise uniqueness of solutions to \eqref{equation}. This, together with a classical argument of Gy\"{o}ngy and Krylov (see \cite{Krylov}) allows us to deduce convergence in probability of the scheme with respect to the initial stochastic basis.

\section{The finite-volume framework}\label{sectiontwo}

\subsection{Admissible finite-volume meshes and notations}\label{mesh}
In order to perform a finite-volume approximation of the variational solution of Problem \eqref{equation} on $[0,T]\times \Lambda$ we need first of all to set a choice for the temporal and spatial discretization. For the time discretization, let $N\in \mathbb{N}^*$ be given. We define the  fixed time step $\Delta t=\frac{T}{N}$ and  divide the interval $[0,T]$ in $0=t_0<t_1<....<t_N=T$ equidistantly with $t_n=n\Delta t$ for all $n\in \{0, ..., N-1\}$.
For the space discretization, we refer to \cite{gal} and consider finite-volume admissible meshes in the sense of the following definition.
\begin{defi}\label{defmesh} (Admissible finite-volume mesh) An admissible finite-volume mesh $\mathcal{T}$ of $\Lambda$ (see Fig.~\ref{fig:notation_mesh}) is given by a family of open polygonal and convex subsets $K$, called \textit{control volumes} of $\Tau$, satisfying the following properties:
\begin{itemize}
\item $\overline{\Lambda}=\bigcup_{K\in\Tau}\overline{K}$.
\item If $K,L\in\Tau$ with  $K\neq L$ then $\operatorname{int}K\cap\operatorname{int}L=\emptyset$.
\item If $K,L\in\Tau$, with $K\neq L$ then either the $1$-dimensional  Lebesgue measure of $\overline{K}\cap \overline{L}$ is $0$ or $\overline{K}\cap \overline{L}$ is the edge of the mesh denoted $\sigma=K|L$  separating the control volumes $K$ and $L$.
\item To each control volume $K\in\Tau$, we associate a point $x_K\in \overline{K}$ (called the center of $K$) such that: If $K,L\in\Tau$ are two neighbouring control volumes the straight line between the centers $x_K$ and $x_L$ is orthogonal to the edge $\sigma=K|L$.
\end{itemize}
\end{defi}

\begin{figure}[htbp!]
\centering
\begin{tikzpicture}[scale=2]

  \clip (-1.2,-0.6) rectangle (1.8,1.3);

  \node[rectangle,fill] (A) at (-1,0.6) {};
  \node[rectangle,fill] (B) at (0,1.2) {};
  \node[rectangle,fill] (C) at (0,-0.2) {};
  \node[rectangle,fill] (D) at (1.5,0.3) {};

  \centre[above right]{(-0.6,0.5)}{xK}{$x_K$};
  \centre[above left]{(0.9,0.5)}{xL}{$x_L$};
  
  \draw[thick] (B)--(C) node [pos=0.7,right] {$\sigma=$\small{$K|L$}};

  \draw[thin,opacity=0.5] (A) -- (B) -- (C) -- (A) ;
  \draw[thin,opacity=0.5] (D) -- (B) -- (C) -- (D);

  \draw[dashed] (xK) -- (xL);
  
  \coordinate (KK) at ($(xK)!0.65!-90:(xL)$);
  \coordinate (LL) at ($(xL)!0.65!90:(xK)$);

  \draw[dotted,thin] (xK) -- (KK);
  \draw[dotted,thin] (xL) -- (LL);

  \draw[|<->|] (KK) -- (LL) node [midway,fill=white,sloped] {$\dkl$};
 
  \coordinate (KAB) at ($(A)!(xK)!(B)$);
  \coordinate (KAC) at ($(A)!(xK)!(C)$);

  \coordinate (LDB) at ($(D)!(xL)!(B)$);
  \coordinate (LDC) at ($(D)!(xL)!(C)$);

  \draw[dashed] (xK) -- ($(xK)!3!(KAB)$);
  \draw[dashed] (xK) -- ($(xK)!3!(KAC)$);

  \draw[dashed] (xL) -- ($(xL)!3!(LDB)$);
  \draw[dashed] (xL) -- ($(xL)!3!(LDC)$);

  \begin{scope}[on background layer]   
    \draw (0,0.5) rectangle ++ (0.1,-0.1);
  \end{scope}
  
  \coordinate (nkl) at ($(B)!0.3!(C)$);
  \draw[->,>=latex] (nkl) -- ($(nkl)!0.3cm!90:(C)$) node[above] {$\mathbf{n}_{K|L}$};

\end{tikzpicture}
\caption{Notations of the mesh $\mathcal T$ associated with $\Lambda$\label{fig:notation_mesh}}
\end{figure}
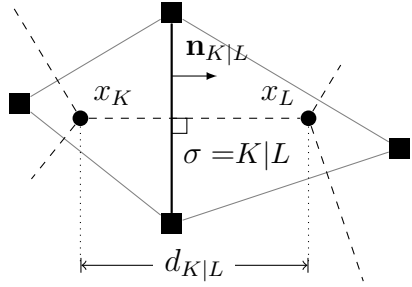

Once an admissible finite-volume mesh $\Tau$ of $\Lambda$ is fixed, we will use the following notations.\\
\quad\\
\textbf{Notations.}
\begin{itemize}
\item $h=\operatorname{size}(\Tau)=\sup\{\operatorname{diam}(K): K\in\Tau\}$, the mesh size.

\item $d_h\in\mathbb{N}$ the number of control volumes $K\in\Tau$ with $h=\operatorname{size}(\Tau)$.

\item $\lambda_2$ denotes the two-dimensional Lebesgue measure.

\item $\mathcal{E}$ is the set of the edges of the mesh $\Tau$ and we define $\mathcal{E}_{\operatorname{int}}:=\{\sigma\in\mathcal{E}:\sigma\nsubseteq \partial\Lambda\}$, $\mathcal{E}_{\operatorname{ext}}:=\{\sigma\in\mathcal{E}:\sigma\subseteq \partial\Lambda\}$.

\item For $K\in\Tau$, $\mathcal{E}_K$ is the set of edges of $K$ and $m_K:=\lambda_2(K)$.

\item Let $K,L\in\Tau$ be two neighbouring control volumes. For $\sigma=K|L\in\edgesint$, let $m_\sigma$ be the length of $\sigma$ and $d_{K|L}$ the distance between $x_K$ and $x_L$.

\item For neighbouring control volumes $K,L\in \Tau$, we denote the unit vector on the edge $\sigma=K|L$ pointing from $K$ to $L$ by $\mathbf{n}_{KL}$. 

\item For $\sigma=K|L\in\edgesint$, the diamond $D_\sigma$ (see Fig.~\ref{fig:notation_diamond}) is the open quadrangle whose diagonals are the edge $\sigma$ and the segment $[x_K,x_L]$.
For $\sigma\in\edgesext\cap\edges_K$, we define $D_\sigma:=K$. Then, $\overline{\Lambda}=\bigcup_{\sigma\in\edges} \overline{D_\sigma}$. 
\item  $m_{D_\sigma}:=\lambda_2(D_\sigma)$ is the two-dimensional Lebesgue measure of the diamond $D_{\sigma}$ . Note that for $\sigma\in\edgesint$, we have $\displaystyle m_{D_\sigma}=\frac{m_\sigma\dkl}{2}.$

\end{itemize}

\begin{figure}[htbp!]
\centering
\begin{tikzpicture}[scale=1.8]

  \clip (-1.2,-0.6) rectangle (1.8,1.6);

  \coordinate (xK) at (-0.5,0.5);
  \coordinate (xL) at (0.8,0.5);

  \node[rectangle,fill] (A) at (-1,1.5) {};
  \node[rectangle,fill] (B) at (0,1) {};
  \node[rectangle,fill] (C) at (0,0) {};
  \node[rectangle,fill] (D) at (-0.5,-0.3) {};
  
  \node[rectangle,fill] (E) at (1,1.2) {};
  \node[rectangle,fill] (F) at (1.2,-0.5) {};

  \centre[left]{(-0.7,0.5)}{xK}{$x_K$};
  \centre[right]{(1.2,0.5)}{xL}{$x_L$};

  \draw[thin,opacity=0.5] (-10,1.5)--(A) -- (B) -- (C) -- (D) -- (-10,-0.5) ;
  \draw[thin,opacity=0.5] (10,2) -- (E) -- (B) -- (C) -- (F) -- (4,0);
  \draw[dashdotted] (xK) -- (C) -- (xL) -- (B) -- (xK);

  \draw[thick] (B) -- (C);

  \draw[thick,dashed] (xK) -- (xL);

  \draw (0,0.5) node [above right]  {$D_\sigma$};
  \draw (0,0.3) node [left]  {$\sigma$};

  \begin{scope}[on background layer]
    \fill [diamond,fill opacity=0.2] (xK.center) -- (C.center)-- (xL.center) -- (B.center);
    \draw (0,0.5) rectangle ++ (0.1,-0.1);
  \end{scope}

\end{tikzpicture}
\caption{Notations on a diamond cell $D_\sigma$ for $\sigma\in\edgesint$\label{fig:notation_diamond}}
\end{figure}

Using these notations, we introduce a positive number 
\begin{eqnarray}\label{mrp}
\reg=\max\left(\mathcal N,\max_{\scriptscriptstyle K \in\Tau \atop \sigma\in\mathcal{E}_K} \frac{\operatorname{diam}(K)}{d(x_K,\sigma)}\right)
\end{eqnarray}
(where $\mathcal N$ is the maximum of edges incident to any vertex) that measures the regularity of a given mesh and is useful to perform the convergence analysis of finite-volume schemes.
This number should be uniformly bounded when the mesh size tends to $0$ for the convergence results to hold.

\subsection{Discrete unknowns and piecewise constant functions }\label{discretenotation}

From now on and unless otherwise specified, we consider $N\in \mathbb{N}^*$, $\Delta t=\frac{T}{N}$ and $\Tau$ an admissible finite-volume mesh of $\Lambda$ in the sense of Definition \ref{defmesh} with a mesh size $h$. For $n\in\{0, ..., N-1\}$ given, the idea of a finite-volume scheme for the approximation of Problem \eqref{equation} is to associate to each control volume $K\in\Tau$ and time $t_n$ a discrete unknown value denoted $u^n_K\in \mathbb{R}$, expected to be an approximation of $u(t_n,x_K)$, where $u$ is the variational solution of \eqref{equation}. Before presenting the numerical scheme satisfied by the discrete unknowns $\{u^n_K, K\in\Tau, n\in\{0, ..., N-1\}\}$, let us introduce some general notations.  \\

For any arbitrary vector $(w_K^n)_{K\in\Tau}\in\re^{d_h}$ we can define the piecewise constant function $w_h^n:\Lambda\rightarrow \mathbb{R}$ by
\[
w_h^n(x):=\sum_{K\in\Tau} w^n_K \mathds{1}_K(x),\ \forall x\in \Lambda.
\]
Note that since the mesh $\Tau$ is fixed, by the continuous mapping defined from $ \mathbb{R}^{d_h}$ to $L^2(\Lambda)$ by 
\[(w^n_K)_{{K\in\Tau}} 
\mapsto \sum_{K\in\Tau}\mathds{1}_K w^n_K, \]
the space $\re^{d_h}$ can be considered as a finite-dimensional subspace of $L^2(\Lambda)$ and we may naturally identify the function and the vector
\[
w^n_h\equiv(w^n_K)_{K\in\Tau}\in \mathbb{R}^{d_h}.
\]
Then, knowing for all $n \in\{0,\ldots,N\}$ the function $w_h^n$,  we can define the following piecewise constant functions in time and space $\whnr,   \whnl :[0,T]\times \Lambda\rightarrow \mathbb{R}$ by  
\begin{equation}
 \label{eq:notation_wh}
 \begin{aligned}  
  \whnr(t,x):=&\sum_{n=0}^{N-1} w_h^{n+1}(x)\mathds{1}_{[t_n,t_{n+1})}(t)\text{ if }t\in[0,T)
\text{ and } \whnr(T,x):=w_h^N(x),\\
\whnl(t,x):=&\sum_{n=0}^{N-1} w_h^n(x)\mathds{1}_{[t_n,t_{n+1})}(t) \text{ if }t\in(0,T]\text{ and } 
\whnl(0,x):=w_h^0(x).
 \end{aligned}
\end{equation}
\begin{remark}
The superscripts $r$ and $l$ in \eqref{eq:notation_wh} do not refer to the continuity properties of the associated functions (which may be chosen either c\`{a}dl\`{a}g or c\`{a}gl\`{a}d). The difference is that $w_{h,N}^l$ is adapted whereas $w_{h,N}^r$ is not adapted.
\end{remark}
As for the piecewise constant function in space, since $\Tau$ and $N$ are fixed, by the continuous mapping defined from $\mathbb{R}^{d_h\times N}$ to $L^2(0,T;L^2(\Lambda))$ by
\[(w_K^n)_{\substack{K\in\Tau \\ n\in\{0,\ldots,N-1\}}}\mapsto\sum_{\substack{K\in\Tau \\ n\in\{0,\ldots,N-1\}}}\mathds{1}_K\mathds{1}_{[t_n,t_{n+1})}w_K^n,\]
the space $\mathbb{R}^{d_h\times N}$ can be considered as a finite-dimensional subspace of $L^2(0,T;L^2(\Lambda))$ and we may naturally identify
\begin{align*}
\whnl&\equiv(w_K^n)_{\substack{K\in\Tau \\ n\in\{0,\ldots,N-1\}}}\in \mathbb{R}^{d_h\times N},\\
\whnr&\equiv(w_K^{n+1})_{\substack{K\in\Tau \\ n\in\{0,\ldots,N-1\}}}\in \mathbb{R}^{d_h\times N}.
\end{align*}
We can also define the piecewise affine, continuous in time and piecewise constant in space reconstruction $\widehat{w}_{h,N}:[0,T]\times \Lambda\rightarrow\re$ by
\begin{equation}
 \label{eq:notation_whhat}
\widehat{w}_{h,N}(t,x):=\sum_{n=0}^{N-1}\mathds{1}_{[t_n,t_{n+1})}(t)
\left(\frac{w_h^{n+1}(x)-w_h^n(x)}{\Delta t}(t-t_n)+w_h^n(x)\right).
\end{equation}
\begin{remark}Note that in the rest of the paper, when we will consider a time and space function $\alpha:[0,T]\times \Lambda\rightarrow \mathbb{R}$ on all the space $\Lambda$ (respectively the time interval $[0,T]$) at a fixed time $t\in [0,T]$ (respectively at a fixed $x\in \Lambda$) we will omit the space (respectively time) variable in the notations and write $\alpha(t)$  (respectively $\alpha(x)$) instead of $\alpha(t,\cdot)$ (respectively $\alpha(\cdot,x)$).
\end{remark}

\subsection{Discrete norms and discrete gradient}

Fix $n\in\{0, ..., N-1\}$ and consider for the remainder of this subsection an arbitrary vector $(w^n_K)_{K\in\Tau}\in \mathbb{R}^{d_h}$ and use its natural identification with the piecewise constant function in space $w^n_h\equiv(w^n_K)_{K\in\Tau}$.  We  introduce in what follows the notions of discrete gradient and discrete norms for such a  function $w^n_h$.
\begin{defi}[Discrete $L^2$-norm]We define the $L^2$-norm  of $w^n_h \in\re^{d_h}$ as follows
$$||w^n_h||_{L^2(\Lambda)}=\left(\sum_{K\in \Tau}m_K |w^n_K|^2\right)^\frac12.$$
\end{defi}

\begin{defi}[Discrete gradient]
We define the gradient operator $\nabla^h$ that maps scalar fields $w^n_h\in\re^{d_h}$ into vector fields of $(\re^2)^{e_h}$ (where $e_h$ is the number of edges in the mesh $\Tau$), we set $\nabla^h w^n_h =(\nabla_\sigma^h w^n_h)_{\sigma\in\edges}$ with
 \[
  \nabla_\sigma^h w^n_h :=
  \left\{
  \begin{aligned}
   2\frac{w^n_L-w^n_K}{\dkl} \mathbf{n}_{KL}, \quad
   &\text{ if }\sigma=K|L\in\edgesint ; \\
   \qquad 0, \qquad\qquad &\text{ if } \sigma\in\edgesext.
  \end{aligned}
  \right.
 \]
\end{defi}

We remark that $\nabla^h w_n^h$ is considered as a piecewise constant function, which is constant on the diamond $D_\sigma$.

\begin{defi}[Discrete $H^1$-seminorm]
We define the $H^1$-seminorm  of $w^n_h \in\re^{d_h}$ as follows
 \[
  |w^n_h|_{1,h}:=\left(\sum_{\sigma\in\edgesint}\frac{m_\sigma}{\dkl}|w^n_K-w^n_L|^2\right)^\halbe.
 \]

\end{defi}

\begin{notation}
If not marked otherwise, for an edge $\sigma\in\edgesint$ we denote the neighbouring control volumes by $K$ and $L$, i.e., $\sigma=K|L$. In particular we use this notation in sums.
\end{notation}

\begin{remark}\label{remarkforuhnrboundiii} 
Note that in particular, 
\begin{align*}
\|\nabla^h w^n_h\|_{(L^2(\Lambda))^2}^2
=\sum_{\sigma\in\edgesint}m_{D_\sigma}|\nabla_\sigma^h w^n_h|^2
=2\sum_{\sigma\in\edgesint}\frac{m_\sigma}{\dkl}|w^n_K-w^n_L|^2=2|w^n_h|_{1,h}^2
\end{align*}
where the constant $2$ corresponds to the space dimension $d=2$.
\end{remark}

\begin{remark}\label{discrpartint}
If we consider another arbitrary vector $\widetilde w^n_h\equiv(\widetilde{w}^n_K)_{K\in\Tau}\in \mathbb{R}^{d_h}$, by summing over the edges we may rearrange the sum on the left-hand side and get the following rule of "discrete partial integration"
\begin{align}\label{PInt}
\sum_{K\in\Tau}\sum_{\sigma\in\edges_K\cap\edgesint}\frac{m_\sigma}{\dkl}(w^n_K-w^n_L)\widetilde w^n_K
=\sum_{\sigma\in\edgesint}\frac{m_\sigma}{\dkl}(w^n_K-w^n_L)(\widetilde w^n_K-\widetilde w^n_L).
\end{align}
\end{remark}
\quad\\
We have now all the necessary  definitions and notations to present the finite-volume scheme studied in this paper. This is the aim of the next subsection.

\subsection{The finite-volume scheme}

Firstly, we define the vector  $u_h^0\equiv (u^0_K)_{K\in\Tau} \in \re^{d_h}$ by the discretization of the initial condition $u_0$ of Problem \eqref{equation} over each control volume:
\begin{align}
\label{eq:def_u0}
u_K^0:=\frac{1}{m_K}\int_K u_0(x)\,dx, \quad \forall K\in \Tau.
\end{align}
The finite-volume scheme we propose reads, for this given initial $\mathcal{F}_0$-measurable random vector $u_h^0\in\re^{d_h}$: \\
For any $n \in \{0,\cdots,N-1\}$, knowing $u_h^n\equiv (u^{n}_K)_{K\in\Tau} \in \re^{d_h}$ we search for $u_h^{n+1}\equiv(u^{n+1}_K)_{K\in\Tau}\in\re^{d_h}$ such that, for almost every $\omega\in\Omega$, the vector $u_h^{n+1}$ is solution to the following random equations
\begin{equation}
\label{equationapprox}
\frac{m_K}{\Delta t}(u_K^{n+1}-u_K^n)+\sum_{\sigma\in\edgesint\cap\edges_K}\frac{m_\sigma}{\dkl}(u_K^{n+1}-u_L^{n+1})=\frac{m_K}{\Delta t}g(u_K^n)\Delta_{n+1}W,
\quad  \forall K \in \Tau,
\end{equation}
where $\Delta_{n+1}W$ denotes the increments of the Brownian motion between $t_{n+1}$ and $t_n$:
$$\Delta_{n+1}W:=W(t_{n+1})-W(t_n)\text{ for }n\in\{0,\dots,N-1\}.$$ 
\begin{remark}
The second term on the left-hand side of \eqref{equationapprox} is the classical two-point flux approximation of the Laplace operator obtained formally by integrating the Laplace operator on each control volume $K \in \Tau$, then applying the Gauss-Green theorem to the term $\int_K \Delta u(t_{n+1},x)dx$ and finally combining  Taylor expansions of the function $u(t_{n+1},\cdot)$ at the points $x_K$ and $x_L$ together with the orthogonality condition on the mesh (see~\cite[Section 10]{gal} for more details on the two-point flux approximation of the Laplace operator with Neumann boundary conditions).
The time-implicit discretization of the Laplace operator has several analytic advantages: First of all, calculations in the a-priori estimates are simplified. Secondly, we omit the use of a CFL-condition. Last but not least, for more general nonlinear operators such as the $p$-Laplace operator, an implicit time discretization is more appropriate. However, an explicit time discretization of the noise is crucial and can not be omited due to the non-anticipative character of the It\^{o} stochastic integral.
\end{remark}

We can note that by multiplying equation~\eqref{equationapprox} by $w_K$, summing over $K\in\Tau$ and using equality~\eqref{PInt}, the numerical scheme can be rewritten as: For any $n\in\{0,\ldots,N-1\}$, find $u_h^{n+1}\in\re^{d_h}$ such that for any $w_h \in \re^{d_h}$,
\begin{align}\label{eq:FV_var}
\begin{aligned}
&\sum_{K\in\Tau} m_K \left(u^{n+1}_K -u_K^n\right) w_K 
+\Delta t \sum_{\sigma\in\edgesint}\frac{m_\sigma}{\dkl} (u^{n+1}_K-u^{n+1}_L)(w_K-w_L) \\
&= \sum_{K\in\Tau} m_K g(u_K^n)w_K\Delta_{n+1}W.
\end{aligned}
\end{align}
The two formulations are equivalent but this "variational" formulation will be more useful in the analysis to follow.

\begin{prop}[Existence of a discrete solution]
\label{210609_prop1}
Assume that hypotheses $H_1$ and $H_2$ hold. Let $\Tau$ be an admissible finite-volume mesh of $\Lambda$  in the sense of Definition \ref{defmesh} with a mesh size $h$ and $N\in \mathbb{N}^*$. Then, there exists a unique solution $(u_h^n)_{1\le n \le N} \in (\re^{d_h})^N$ to Problem~\eqref{equationapprox} associated with the initial vector $u^0_h$ defined by~\eqref{eq:def_u0}. Additionally, for any $n\in \{0,\ldots,N\}$, $u_h^n$ is a $\mathcal{F}_{t_n}$-measurable random vector.
\end{prop}

The solution $(u_h^n)_{1\le n \le N} \in (\re^{d_h})^N$ of the scheme \eqref{eq:def_u0}-\eqref{equationapprox} is then used to build the right and left finite-volume approximations $u^r_{h,N}$ and $u^l_{h,N}$ defined by \eqref{eq:notation_wh} for the variational solution $u$ of Problem \eqref{equation}.

\begin{proof}(Proposition \ref{210609_prop1}).
Set $n\in\{0,\ldots,N-1\}$. For $K\in\Tau$ and a.s. in $\Omega$, note that~\eqref{eq:FV_var} can be rewritten in the following way:
\begin{equation}\label{210609_01}
\sum_{K\in\Tau} m_K \left(u^{n+1}_K -f_K^n\right) w_K 
+\Delta t \sum_{\sigma\in\edgesint}\frac{m_\sigma}{\dkl} (u^{n+1}_K-u^{n+1}_L)(w_K-w_L)=0,
\end{equation}
where $f_K^n:=g(u_K^n)\Delta_{n+1}W +u_K^n$.
For $f_h^n\equiv(f_K^n)_{K\in\Tau}\in\mathbb{R}^{d_h}$ and a.e. $\omega\in \Omega$, we can define the functional $J_h^n:\mathbb{R}^{d_h}\rightarrow \mathbb{R}$ by 
\[
 J_h^n(w_h)= \frac12 a(w_h,w_h)- \int_{\Lambda}w_h f_h^n \,dx
 \]
 where the bilinear form $a:\mathbb{R}^{d_h}\times\mathbb{R}^{d_h}\rightarrow \mathbb{R}$ is given by
 \[
 a(v_h,w_h)=\int_{\Lambda}v_h w_h\,dx +\Delta t\sum_{\sigma\in\edgesint}\frac{m_\sigma}{\dkl} (u_K-u_L)(w_K-w_L).
\]
From a straightforward calculation it is easy to see that the bilinear form $a$ is symmetric, continuous and coercive.

Thus from the Theorem of Stampacchia (see e.g. \cite[Theorem 5.6]{BrezisFA}), $J_h^n$ admits a unique minimizer $u_h^{n+1}\in\mathbb{R}^{d_h}$ and the associated sequence $(u_h^n)_{1\le n \le N} \in (\re^{d_h})^N$  is the unique solution of \eqref{210609_01} a.s. in $\Omega$.
If we assume that $u_h^n$ is $\mathcal{F}_{t_n}$-measurable,  then $f_h^n$ is $\mathcal{F}_{t_{n+1}}$-measurable and consequently the random variable $\omega\mapsto J_h^n(w_h)(\omega)$ is $\mathcal{F}_{t_{n+1}}$-measurable for any $w_h\in\mathbb{R}^{d_h}$.
Hence,  $\omega\mapsto u_h^{n+1}(\omega)=\min_{w_h\in\mathbb{R}^{d_h}} J_h^n(w_h)(\omega)$ is $\mathcal{F}_{t_{n+1}}$-measurable.
Then, it follows by iteration, that for a given, $\mathcal{F}_0$-measurable random variable $u^0_h\in\mathbb{R}^{d_h}$, for any $n\in \{0,\ldots, N-1\}$ there exists a $\mathcal{F}_{t_{n+1}}$-measurable function $u_h^{n+1}\in\mathbb{R}^{d_h}$ such that $(u_h^n)_{1\le n \le N} \in (\re^{d_h})^N$ is solution to Problem \eqref{equationapprox} associated with the initial vector $u^0_h$ which concludes the proof.
\end{proof}

\section{Stability estimates}\label{estimates}

We will derive in this section several stability estimates satisfied by the discrete solution $(u_h^n)_{1 \le n \le N} \in (\re^{d_h})^N$ of the scheme \eqref{eq:def_u0}-\eqref{equationapprox} given by Proposition \ref{210609_prop1}, and also by the associated right and left finite-volume approximations $u^r_{h,N}$ and $u^l_{h,N}$ defined by \eqref{eq:notation_wh}.

\subsection{Bounds on the finite-volume approximations}
We start by giving a bound on the discrete initial data.
\begin{lem}
\label{bound_u0}
Let $u_0$ be a given function satisfying assumption $H_1$. Then, the associated discrete initial data $u_h^0 \in \re^{d_h}$ defined by~\eqref{eq:def_u0} satisfies $\mathds{P}$-a.s. in $\Omega$,
\begin{equation*}
\|u_h^0\|_{L^2(\Lambda)} \leq \|u_0\|_{L^2(\Lambda)}.
\end{equation*}
\end{lem}
The proof is a direct consequence of the definition of $u_h^0$ and the Cauchy-Schwarz inequality.\\

We can now give the bounds on the discrete solutions which is one of the key points of the proof of the convergence theorem.

\begin{prop}[Bounds on the discrete solutions]
\label{bounds}
There exists a constant $C_1>0$, depending only on $u_0$, $C_L$, $|\Lambda|$ and $T$ such that
\begin{align*}
&\erwb \|u_h^n\|_{L^2(\Lambda)}^2 \erwe+\erwb\sum_{k=0}^{n-1}\|u_h^{k+1}-u_h^k\|_{L^2(\Lambda)}^2\erwe+2\Delta t \sum_{k=0}^{n-1}\erww{|u_h^{k+1}|_{1,h}^2}\leq C_1,\; \forall n\in \{1,\ldots,N\}.
\end{align*}
\end{prop}

\begin{proof}
We fix $n\in \{1,\ldots,N\}$. For any $k\in \{0,\ldots,n-1\}$, choosing $w_h=u_h^{k+1}$ as test function in~\eqref{eq:FV_var} we obtain,
\begin{align}\label{implicitscheme}
\begin{split}
&\sum_{K\in\Tau}\frac{m_K}{\Delta t}(u_K^{k+1}-u_K^k)u_K^{k+1}
+ \sum_{\sigma\in\edgesint}\frac{m_\sigma}{\dkl}|u_K^{k+1}-u_L^{k+1}|^2\\
&=\sum_{K\in\Tau}\frac{m_K}{\Delta t}g(u_K^k)u_K^k\Delta_{k+1}W+\sum_{K\in\Tau}\frac{m_K}{\Delta t}g(u_K^k)(u_K^{k+1}-u_K^k)\Delta_{k+1}W.
\end{split}
\end{align}
We consider the terms separately: For the first term on the left-hand side we find
\begin{align*}
\sum_{K\in\Tau}\frac{m_K}{\Delta t}(u_K^{k+1}-u_K^k)u_K^{k+1}=\halbe\sum_{K\in\Tau}\frac{m_K}{\Delta t}(|u_K^{k+1}|^2-|u_K^k|^2+|u_K^{k+1}-u_K^k|^2).
\end{align*}
Taking expectation in \eqref{implicitscheme}, the first expression on the right-hand side of \eqref{implicitscheme} vanishes, since $u_K^k$ and $\Delta_{k+1}W$ are independent and therefore
\[\erwb g(u_K^k)u_K^k\Delta_{k+1}W\erwe=0.\]
In the second term we apply Young's inequality in order to keep all necessary terms. Then, taking expectation and using the It\^{o} isometry we obtain
\begin{align*}
\erwb g(u_K^k)(u_K^{k+1}-u_K^k)\Delta_{k+1}W\erwe&\leq\erwb|g(u_K^k)\Delta_{k+1}W|^2\erwe+\frac{1}{4}\erwb|u_K^{k+1}-u_K^k|^2\erwe\\
&\leq \Delta t\erwb|g(u_K^k)|^2\erwe+\frac{1}{4}\erwb|u_K^{k+1}-u_K^k|^2\erwe
\end{align*}
for any $K\in \Tau$. Altogether we find
\begin{align*}
&\frac{1}{2\Delta t}\int_\Lambda\erwb |u_h^{k+1}|^2-|u_h^k|^2\erwe dx+\frac{1}{4\Delta t}\int_\Lambda\erwb|u_h^{k+1}-u_h^k|^2\erwe dx+\erwb|u_h^{k+1}|_{1,h}^2\erwe\\
&\leq \int_\Lambda\erwb|g(u_h^k)|^2\erwe dx.
\end{align*}
Summing over $k\in\{0,\dots,n-1\}$ and multiplying with $2\Delta t$ we obtain
\begin{align}\label{lem1beschr}
\begin{split}
&\erwb\|u_h^n\|_{L^2(\Lambda)}^2-\|u_h^0\|_{L^2(\Lambda)}^2\erwe+\frac{1}{2}\sum_{k=0}^{n-1}\erww{\|u_h^{k+1}-u_h^k\|_{L^2(\Lambda)}^2}+ 2\Delta t\sum_{k=0}^{n-1}\erwb|u_h^{k+1}|_{1,h}^2\erwe\\
&\leq 2\Delta t\sum_{k=0}^{n-1}\erwb\|g(u_h^k)\|_{L^2(\Lambda)}^2\erwe.
\end{split}
\end{align}
Since the second and third term in \eqref{lem1beschr} are nonnegative, from $H_2$ and \eqref{H3} it follows that
\[
\erwb\|u_h^n\|_{L^2(\Lambda)}^2\erwe\leq\erwb\|u_h^0\|_{L^2(\Lambda)}^2\erwe+2C_L\Delta t\sum_{k=0}^{n-1}\erwb\|u_h^k\|_{L^2(\Lambda)}^2\erwe+2C_L|\Lambda|T.
\]
Applying the discrete Gronwall lemma yields
\begin{align}\label{uhnbound}
\erwb\|u_h^n\|_{L^2(\Lambda)}^2\erwe\leq\left((1+2C_L T)\erwb\|u_h^0\|_{L^2(\Lambda)}^2\erwe+2C_L|\Lambda|T\right)e^{2C_LT}.
\end{align}
From \eqref{uhnbound} and Lemma~\ref{bound_u0} we may conclude that there exists a constant $\Upsilon>0$ such that
\begin{align*}
\sup_{n\in\{1,\dots,N\}}\erww{\|u_h^n\|_{L^2(\Lambda)}^2}\leq \Upsilon.
\end{align*}
Applying \eqref{uhnbound}, $H_2$ and \eqref{H3} it follows that
\begin{align}\label{210819_02}
2\Delta t\sum_{k=0}^{n-1}\erww{\|g(u_h^k)\|_{L^2(\Lambda)}^2}\leq 2C_L|\Lambda|T+2C_L\Delta t\sum_{k=0}^{n-1}\erww{\|u_h^k\|_{L^2(\Lambda)}^2}\leq 2C_LT(\Upsilon+|\Lambda|)
\end{align}
for all $n\in\{1,\ldots N\}$. From \eqref{lem1beschr}, Lemma~\ref{bound_u0} and \eqref{210819_02} it now follows that
\begin{align*}
&\erwb\|u_h^n\|_{L^2(\Lambda)}^2\erwe+\frac{1}{2}\sum_{k=0}^{n-1}\erww{\|u_h^{k+1}-u_h^k\|_{L^2(\Lambda)}^2}+ 2\Delta t\sum_{k=0}^{n-1}\erwb|u_h^{k+1}|_{1,h}^2\erwe\\
&\leq \erwb\|u_0\|_{L^2(\Lambda)}^2\erwe+2C_LT(\Upsilon+|\Lambda|)=:C_1
\end{align*}
for all $n\in \{1,\ldots, N\}$.
\end{proof}

We are now interested in the bounds on the right and left finite-volume approximations defined by~\eqref{eq:notation_wh}.
As a direct consequence of Proposition~\ref{bounds} we get a $L^2(\Omega;L^2(0,T;\lzlambda))$-bound on these approximations.
\begin{lem}\label{210611_lem01}
The sequences $(\uhnr)_{h,N}$ and $(\uhnl)_{h,N}$ are bounded independently of the discretization parameters $N\in\mathbb{N}^{\ast}$ and $h$ in $L^2(\Omega;L^2(0,T;\lzlambda))$.
\end{lem}

Thanks to Proposition~\ref{bounds} we can also obtain a $L^2(\Omega;L^2(0,T;L^2(\Lambda)))$-bound on the discrete gradients of the finite-volume approximations.

\begin{lem}\label{remarkuhnrboundintomega}
There exist a constant $\ctel{K1}\geq 0$ depending only on $u_0$, $C_L$, $|\Lambda|$ and $T$ and a constant $\ctel{K2}\geq 0$ additionally depending on the mesh regularity $\reg$ (defined by \eqref{mrp}), such that
\begin{align}\label{uhnrboundinotomega}
\int_0^T\erwb|\uhnr(t)|_{1,h}^2\erwe dt\leq \cter{K1}
\end{align}
and
\begin{align}\label{210611_01}
\int_0^T\erwb|\uhnl(t)|_{1,h}^2\erwe dt\leq \cter{K2}.
\end{align}
\end{lem}

\begin{proof}
\[\int_0^T\erwb|\uhnr(t)|_{1,h}^2\erwe dt=\Delta t\sum_{k=0}^{N-1}\erwb|u_h^{k+1}|_{1,h}^2\erwe\]
and therefore \eqref{uhnrboundinotomega} follows directly from Proposition~\ref{bounds}. Using the definition of $\uhnl$ and \eqref{uhnrboundinotomega}, we get
\begin{align*}
\int_0^T\erwb|\uhnl(t)|_{1,h}^2\erwe dt
&\leq \Delta t\erwb |u_h^0|_{1,h}^2\erwe+\Delta t \sum_{k=0}^{N-1}\erwb|u_h^{k+1}|_{1,h}^2\erwe
\leq \Delta t\erwb|u_h^0|_{1,h}^2\erwe+\cter{K1}.
\end{align*}
Since $u_0$ is assumed to be in $L^2(\Omega;H^1(\Lambda))$, from \cite[Lemma 9.4]{gal}, it follows that there exists $C_{\Lambda}\geq 0$ depending on the mesh regularity $\reg$ such that,
\begin{equation*}
\erwb |u_h^0|_{1,h}^2\erwe\leq C_{\Lambda}\erwb \|\nabla u_0\|_{L^2(\Lambda)}^2\erwe
\end{equation*}
and therefore \eqref{210611_01} follows.
\end{proof}

We end this section by a bound on the discrete solution which will be useful for obtaining the time translate estimate and the bounds on the Gagliardo seminorm.
Note that the difficulty here is to have the maximum inside the expectation.

\begin{lem}\label{uhnsupbound}
There exists a constant $\ctel{K4}\geq 0$ independent of the discretization parameters $N\in\mathbb{N}^{\ast}$ and $h$, such that
\begin{align*}
\erwb\max_{n\in\{0,\dots,N\}}\|u_h^n\|_{L^2(\Lambda)}^2\erwe\leq \cter{K4}.
\end{align*}
\end{lem}

\begin{proof}
For $N\in\mathbb{N}$, we choose an arbitrary $k\in\{0,\dots,N-1\}$ and an arbitrary $K\in\Tau$. Testing the implicit scheme \eqref{eq:FV_var} with $u_K^{k+1}$ yields
\begin{align*}
\sum_{K\in\Tau}\frac{m_K}{\Delta t}(u_K^{k+1}-u_K^k)u_K^{k+1}+\sum_{\sigma\in\edgesint}\frac{m_\sigma}{\dkl}|u_K^{k+1}-u_L^{k+1}|^2=\sum_{K\in\Tau}\frac{m_K}{\Delta t}g(u_K^{k})u_K^{k+1}\Delta_{k+1}W.
\end{align*}
This provides by Cauchy-Schwarz and Young inequalities
\begin{align*}
&\halbe\left(\|u_h^{k+1}\|_{L^2(\Lambda)}^2-\|u_h^k\|_\lzlambda^2+\|u_h^{k+1}-u_h^k\|_\lzlambda^2\right)\\
&\leq\left(\int_{t_k}^{t_{k+1}} g(u_h^k)dW(s), u_h^{k+1}-u_h^k\right)_{L^2(\Lambda)}+\left(\int_{t_k}^{t_{k+1}}g(u_h^k)dW(s),u_h^k\right)_\lzlambda\\
&\leq\halbe\left\|\int_{t_k}^{t_{k+1}} g(u_h^k)dW(s)\right\|_\lzlambda^2+\halbe\left\|u_h^{k+1}-u_h^k\right\|_\lzlambda^2+\left(\int_{t_k}^{t_{k+1}}g(u_h^k)dW(s),u_h^k\right)_\lzlambda.
\end{align*}
We obtain
\begin{align*}
\|u_h^{k+1}\|_\lzlambda^2-\|u_h^k\|_\lzlambda^2\leq\left\|\int_{t_k}^{t_{k+1}}g(u_h^k)dW(s)\right\|_\lzlambda^2+2\left(\int_{t_k}^{t_{k+1}}g(u_h^k)dW(s),u_h^k\right)_\lzlambda.
\end{align*}
For $n\in \{1,\dots,N\}$ fixed, we sum over $k=\{0,\dots,n-1\}$ to obtain
\begin{align*}
\|u_h^{n}\|_\lzlambda^2
\leq \sum_{k=0}^{n-1}\left\|\int_{t_k}^{t_{k+1}}g(u_h^k)dW(s)\right\|_\lzlambda^2 \!\!
+ 2\sum_{k=0}^{n-1}\left(\int_{t_k}^{t_{k+1}}g(u_h^k)dW(s),u_h^k\right)_\lzlambda \!\!
+ \|u_h^0\|_\lzlambda^2.
\end{align*}
From this, taking the maximum over $n\in \{1,\ldots, N\}$ first and then the expectation applying It\^{o} isometry it follows that
\begin{align}
\label{eqmaxuhnl}
\begin{aligned}
\mathbb{E}\Big[&\max_{n=1,\dots,N}\|u_h^n\|_\lzlambda^2\Big]\\
\leq& \erwb\sum_{k=0}^{N-1}\left\|\int_{t_k}^{t_{k+1}}g(u_h^k)dW(s)\right\|_\lzlambda^2 \erwe \\ 
&+ 2\erwb\max_{n=1,\dots,N}\sum_{k=0}^{n-1}\left(\int_{t_k}^{t_{k+1}}g(u_h^k)dW(s),u_h^k\right)_\lzlambda \erwe + \erwb\|u_h^0\|_\lzlambda^2\erwe\\
\leq& \sum_{k=0}^{N-1}\int_{t_k}^{t_{k+1}}\erwb\|g(u_h^k)\|_\lzlambda^2\erwe ds \\
&+2\erwb\max_{n=1,\dots,N}\sum_{k=0}^{n-1}\int_{t_k}^{t_{k+1}}(g(u_h^k),u_h^k)_\lzlambda dW(s)\erwe+\erww{\|u_h^0\|_\lzlambda^2}.
\end{aligned}
\end{align}
We can estimate the second term by the Burkholder-Davis-Gundy inequality
\begin{align*}
&2\erwb\max_{n=1,\dots,N}\sum_{k=0}^{n-1}\int_{t_k}^{t_{k+1}}(g(u_h^k),u_h^k)_\lzlambda dW(s)\erwe\\
&\leq 2\erwb\sup_{t\in[0,T]}\left|\int_0^t(g(\uhnl(s)),\uhnl(s))_\lzlambda dW(s)\right|\erwe\\
&\leq 2 C_B\erwb\left(\int_0^T|(g(\uhnl(s)),\uhnl(s))_\lzlambda|^2ds\right)^\halbe\erwe.
\end{align*}
Now we apply Cauchy-Schwarz and Young inequalities (with $\alpha>0$), $H_2$ and \eqref{H3} to estimate
\begin{align*}
&2 C_B\erwb\left(\int_0^T|(g(\uhnl(s)),\uhnl(s))_\lzlambda|^2ds\right)^\halbe\erwe\\
&\leq 2 C_B\erwb\left(\sup_{t\in[0,T]}\|\uhnl(t)\|_\lzlambda^2\int_0^T\|g(\uhnl(s))\|_\lzlambda^2 ds\right)^\halbe\erwe\\
&\leq 2C_B\erwb\frac{\alpha}{2}\sup_{t\in[0,T]}\|\uhnl(t)\|_\lzlambda^2+\frac{1}{2\alpha}\int_0^T\|g(\uhnl(s))\|_\lzlambda^2 ds \erwe\\
&\leq C_B\alpha\erwb\max_{n=1,\dots,N}\|u_h^n\|_\lzlambda^2\erwe+\frac{C_BC_L}{\alpha}\left(T|\Lambda|+\erwb\int_0^T\|\uhnl(s)\|_\lzlambda^2 ds\erwe\right).
\end{align*}
Plugging the above estimate in \eqref{eqmaxuhnl} and again using $H_2$ with \eqref{H3}, we arrive at
\begin{align*}
\erwb\max_{n=1,\dots,N}\|u_h^n\|_\lzlambda^2\erwe&\leq C_B\alpha\erwb\max_{n=1,\dots,N}\|u_h^n\|_\lzlambda^2\erwe+\erww{\|u_h^0\|_\lzlambda^2}\\
&+C_L\left(\frac{C_B}{\alpha}+1\right)\int_{0}^{T}\erwb\|\uhnl(s)\|_\lzlambda^2\erwe ds\\
&+C_L|\Lambda|T\left(\frac{C_B}{\alpha}+1\right).
\end{align*}
Choosing $\alpha>0$ such that $1-C_B\alpha>0$, we find a constant $C(\alpha,L)>0$ such that
\begin{align*}
\erwb\max_{n=1,\dots,N}\|u_h^n\|_\lzlambda^2\erwe\leq C(\alpha,L)\left(\int_0^T\erww{\|\uhnl(s)\|_\lzlambda^2}\,ds+\erww{\|u_h^0\|_\lzlambda^2}+1\right).
\end{align*}
Now, the assertion follows by Lemmas~\ref{bound_u0} and \ref{210611_lem01}.
\end{proof}

\subsection{Time and space translate estimates}

For the stochastic compactness argument in Subsection \ref{S1}, we need a uniform bound on $(u_{h,N}^l)_{h,N}$ in the spaces $L^2(\Omega;L^2(0,T;W^{\alpha,2}(\Lambda)))$ and $L^2(\Omega;W^{\alpha,2}(0,T;L^2(\Lambda)))$ for $\alpha\in (0,\frac{1}{2})$. In order to prove the bound in $L^2(\Omega;L^2(0,T;W^{\alpha,2}(\Lambda)))$, we establish a uniform estimate on the space translates of $(u_{h,N}^l)_{h,N}$ in Lemma \ref{210112_lem01}. The proof of the bound in $L^2(\Omega;W^{\alpha,2}(0,T;L^2(\Lambda)))$ is more complicated.
To do this, we introduce the following intermediate quantity. For any $(t,x)\in[0,T]\times \Lambda$, we define
\begin{equation}
 \label{eq:defM}
M_{h,N}(t,x):=\int_0^tg(\uhnl(s,x))dW(s).
\end{equation}
Then, Lemma \ref{lemzeittranslate} is a technical result for the proof of Lemma \ref{lemdefMhN}, where we show a uniform estimate on time translates of $(u_{h,N}^l-M_{h,N})_{h,N}$. Thanks to Lemma \ref{lemdefMhN}, we may conclude a uniform bound on $(u_{h,N}^l-M_{h,N})_{h,N}$ in $L^2(\Omega;W^{\alpha,2}(0,T;L^2(\Lambda)))$ in Lemma \ref{uhnlMhNbound}. Then the desired bound on $(u_{h,N}^l)_{h,N}$ is obtained in Lemma \ref{timeboundedness} by using the additional information that $(M_{h,N})_{h,N}$ is bounded in $L^2(\Omega;W^{\alpha,2}(0,T;L^2(\Lambda)))$.\\

We start by giving an estimation of the space translate. We do not give the proof here as it is similar to the one given in~\cite[Theorem 10.3]{gal}.
\begin{lem}\label{210112_lem01}
Let $\buhnl$ be $d\mathds{P}\otimes dt\otimes dx$-a.s. defined by $\buhnl=\uhnl$ on $\Omega\times (0,T)\times \Lambda$ and $\buhnl=0$ on $\Omega\times(\mathbb{R}^3\setminus ((0,T)\times\Lambda))$.
Then there exists a constant $C\geq 0$, only depending on $\Lambda$, such that for all $\eta\in\re^2$ and almost every $t\in[0,T]$ and $\mathds{P}$-a.s in $\Omega$
\begin{align*}
\int_{\re^2}|\buhnl(t,x+\eta)-\buhnl(t,x)|^2 dx\leq C|\eta|
\left(|\uhnl(t)|_{1,h}^2+\|\uhnl(t)\|_{L^2(\Lambda)}^2\right).
\end{align*}
\end{lem}

\begin{lem}\label{lemzeittranslate}
There exists a constant $\ctel{K3}>0$, independent of the discretization parameters $N\in\mathbb{N}^{\ast}$ and $h$, such that for all $\tau\in (0, T)$ there holds
\begin{align}\label{zeittranslate}
\erwb \int_0^{T-\tau}\left\|\uhnl (t+\tau)- M_{h,N}^l(t+\tau)-(\uhnl(t)-M_{h,N}^l(t))\right\|_{L^2(\Lambda)}^2 dt\erwe\leq \cter{K3} \tau,
\end{align}
where $M_{h,N}^l$ is defined for any $(t,x)\in [0,T]\times\Lambda$ by
\begin{align*}
M_{h,N}^l(t,x):=\int_0^{t_n} g(\uhnl(s,x))\, dW(s)\text{ if } t\in[t_n, t_{n+1}) \text{ with } n\in\{0,\ldots, N-1\}.
\end{align*}
\end{lem}

\begin{proof}
Let $\tau\in(0,T)$ be fixed. In the following, we set $\varphi_h^N(t,x):=\uhnl(t,x)-M_{h,N}^l(t,x)$ for $t\in[0,T], x\in\Lambda$. Note that by definition $M_K^n:=M_{h,N}^l(t_n,x_K)$ and $\varphi_K^n:=u_K^n-M_K^n $ for $n\in\{0,\dots,N\},K\in\Tau$. 
For $t\in(0,T-\tau)$ let $n_0(t),n_1(t)\in\{0,\dots,N-1\}$ be the unique nonnegative integer satisfying 
\begin{equation*}
n_0(t)\Delta t\leq t <(n_0(t)+1)\Delta t
\quad \text{and} \quad
n_1(t)\Delta t\leq t+\tau <(n_1(t)+1)\Delta t.
\end{equation*}
There holds $\mathds{P}$-a.s in $\Omega$
\begin{align*}
&\int_0^{T-\tau}\|\uhnl (t+\tau)- M_{h,N}^l(t+\tau)-(\uhnl(t)-M_{h,N}^l(t))\|_{L^2(\Lambda)}^2 dt\\
&=\int_0^{T-\tau}\sum_{K\in\Tau} m_K|\varphi_K^{n_1(t)}-\varphi_K^{n_0(t)}|^2 dt=:\int_0^{T-\tau} A(t)dt.
\end{align*}
Since $\tau>0$, we necessarily have $n_0(t)\leq n_1(t)$. If $n_0(t)=n_1(t)$ holds, we have $A(t)=0$. So we only consider $t\in(0,T-\tau)$ with $n_1(t)>n_0(t)$. Using the notation $\chi_{n+1}(t,t+\tau)=1$ if $(n+1)\Delta t\in[t,t+\tau)$ and $\chi_{n+1}(t,t+\tau)=0$ otherwise for $n=0,\ldots,N-1$, we get
\begin{align*}
A(t)
&=\sum_{K\in\Tau} m_K(\varphi_K^{n_1(t)}-\varphi_K^{n_0(t)})(\varphi_K^{n_1(t)}-\varphi_K^{n_0(t)})\\
&=\sum_{K\in\Tau}m_K(\varphi_K^{n_1(t)}-\varphi_K^{n_0(t)})\sum_{n=n_0(t)}^{n_1(t)-1}(\varphi_K^{n+1}-\varphi_K^n)\\
&=\sum_{K\in\Tau} m_K(\varphi_K^{n_1(t)}-\varphi_K^{n_0(t)})\sum_{n=0}^{N-1}\chi_{n+1}(t,t+\tau)(\varphi_K^{n+1}-\varphi_K^n)\\
&=\sum_{n=0}^{N-1}\chi_{n+1}(t,t+\tau)\sum_{K\in\Tau}(\varphi_K^{n_1(t)}-\varphi_K^{n_0(t)})m_K(\varphi_K^{n+1}-\varphi_K^n).
\end{align*}
Using \eqref{equationapprox}, we obtain
\begin{align*}
A(t)=-\Delta t\sum_{n=0}^{N-1}\chi_{n+1}(t,t+\tau)\sum_{K\in\Tau}(\varphi_K^{n_1(t)}-\varphi_K^{n_0(t)})\sum_{\sigma\in\edgesint\cap\edges_K}\frac{m_\sigma}{\dkl}(u_K^{n+1}-u_L^{n+1}).
\end{align*}
Rearranging the sum in the same way as for discrete partial integration (see Remark~\ref{discrpartint}), using the definition of $\varphi_{h}^N$ and the notation $u_K^{N,l}:=\uhnl(x_K)$ for $K\in\Tau$ one obtains
\begin{align*}
A(t)&=-\Delta t\sum_{n=0}^{N-1}\chi_{n+1}(t,t+\tau)\sum_{\sigma\in\edgesint}\frac{m_\sigma}{\dkl}(u_K^{n+1}-u_L^{n+1})\left(\varphi_K^\net-\varphi_L^\net-(\varphi_K^\nnt-\varphi_L^\nnt)\right)\\
&=-\Delta t\sum_{n=0}^{N-1}\chi_{n+1}(t,t+\tau)\sum_{\sigma\in\edgesint}\frac{m_\sigma}{\dkl}(u_K^{n+1}-u_L^{n+1})[u_K^\net-u_L^\net-u_K^\nnt+u_L^\nnt]\\
&\quad+\Delta t\sum_{n=0}^{N-1}\chi_{n+1}(t,t+\tau)\sum_{\sigma\in\edgesint}\frac{m_\sigma}{\dkl}(u_K^{n+1}-u_L^{n+1})\int_{\nnt\Delta t}^{\net\Delta t} \left(g(u_K^{N,l})-g(u_L^{N,l})\right) dW(s)\\
&=: A_1(t)+A_2(t),
\end{align*}
By Cauchy-Schwarz and Young inequalities we get
\begin{align*}
A_1(t)
&\leq\frac{\Delta t}{2}\sum_{n=0}^{N-1}\chi_{n+1}(t,t+\tau)|u_h^{n+1}|_{1,h}^2+\frac{\Delta t}{2}\sum_{n=0}^{N-1}\chi_{n+1}(t,t+\tau)|u_h^{n_1(t)}-u_h^{n_0(t)}|_{1,h}^2\\
&\leq \frac{\Delta t}{2}\sum_{n=0}^{N-1}\chi_{n+1}(t,t+\tau)|u_h^{n+1}|_{1,h}^2+\Delta t\sum_{n=0}^{N-1}\chi_{n+1}(t,t+\tau)(|u_h^{n_1(t)}|^2_{1,h}+|u_h^{n_0(t)}|_{1,h}^2).
\end{align*}
Consequently,
\begin{align*}
\erwb\int_0^{T-\tau} A_1(t)\,dt\erwe\leq I_1+I_2
\end{align*}
where
\begin{align*}
I_1&=\frac{1}{2}\int_0^{T-\tau}\sum_{n=0}^{N-1}\chi_{n+1}(t,t+\tau)\erwb\Delta t|u^{n+1}_h|^2_{1,h}\erwe\,dt
\end{align*}
and
\begin{align*}
I_2=\int_0^{T-\tau}\sum_{n=0}^{N-1}\chi_{n+1}(t,t+\tau)\Delta t\erwb|u^{n_0(t)}_h|^2_{1,h}+|u^{n_1(t)}_h|^2_{1,h}\erwe\,dt.
\end{align*}
Since
\begin{align*}
\chi_{n+1}(t,t+\tau)=1\quad&\Leftrightarrow\quad (n+1)\Delta t\in[t,t+\tau)\\
&\Leftrightarrow\quad t-\tau\leq (n+1)\Delta t-\tau< t \leq (n+1)\Delta t,
\end{align*}
there arises
\begin{align}\label{chi}
\int_0^{T-\tau}\chi_{n+1}(t,t+\tau)dt=\int_{(n+1)\Delta t-\tau}^{(n+1)\Delta t}1\,dt=\tau.
\end{align}
Using this and \eqref{uhnrboundinotomega}, we have
\begin{align*}
I_1=\frac{\tau}{2}\erwb\int_0^T|u_{h,N}^{r}(s)|^2_{1,h}\,ds\erwe\leq \frac{\cter{K1}\tau}{2}.
\end{align*}
Now we write $I_2=I_{2,1}+I_{2,2}$ where, 
\begin{align*}
I_{2,1}&=\int_0^{T-\tau}\sum_{n=0}^{N-1}\chi_{n+1}(t,t+\tau)\Delta t\erwb|u^{n_0(t)}_h|^2_{1,h}\erwe\,dt,\\
I_{2,2}&=\int_0^{T-\tau}\sum_{n=0}^{N-1}\chi_{n+1}(t,t+\tau)\Delta t\erwb|u^{n_1(t)}_h|^2_{1,h}\erwe\,dt.
\end{align*}
We note that, for any $m\in \{0,\ldots, N-1\}$, if $t\in [t_m,t_{m+1})$ then the definition of $n_0$ implies $n_0(t)=m$ and therefore
\begin{align*}
I_{2,1}
&\leq\sum_{m=0}^{N-1}\left(\int_{t_m}^{t_{m+1}}\sum_{n=0}^{N-1}\chi_{n+1}(t,t+\tau)\,dt\right)\Delta t\erwb|u^{m}_h|^2_{1,h}\erwe.
\end{align*}
Now, we proceed as in \cite[Lemma 6.2]{GHL}. For the convenience of the reader we include the derivation of the formula using the notation of this paper. For all $m \in \{0,\ldots,N-1\}$ we have
\begin{align*}
\int_{t_m}^{t_{m+1}}\sum_{n=0}^{N-1}\chi_{n+1}(t,t+\tau)\,dt=
\sum_{n=0}^{N-1}\int_{t_m-t_{n+1}}^{t_{m+1}-t_{n+1}}
\chi_{n+1}(t+t_{n+1},t+t_{n+1}+\tau)\,dt.
\end{align*}
Now,
\begin{align*}
\chi_{n+1}(t+t_{n+1},t+t_{n+1}+\tau)=1&\Leftrightarrow \ (n+1)\Delta t=t_{n+1}\in [t+t_{n+1},t+t_{n+1}+\tau)\\
&\Leftrightarrow\ t\in(-\tau,0].
\end{align*}
Hence,
\begin{align*}
\int_{t_m}^{t_{m+1}}\sum_{n=0}^{N-1}\chi_{n+1}(t,t+\tau)\,dt\leq \int_{\mathbb{R}}\mathds{1}_{(-\tau,0]}(t)\,dt=\tau 
\quad \forall m \in \{0,\ldots,N-1\}.
\end{align*}
Therefore, thanks to Lemma~\ref{remarkuhnrboundintomega} one has
\begin{align*}
I_{2,1}\leq \tau\Delta t\sum_{m=0}^{N-1}\erwb|u^{m}_h|^2_{1,h}\erwe
=\tau \int_0^T \erwb|u_{h,N}^l(t)|^2_{1,h}\erwe dt\leq \cter{K2}\tau.
\end{align*}
Analogously, for any $m\in\{0,\ldots,N-1\}$, if $t\in [t_m-\tau,t_{m+1}-\tau)$, then the definition of $n_1$ implies $n_1(t)=m$ and therefore
\begin{align*}
I_{2,2}\leq \sum_{m=0}^{N-1}\Delta t \erwb|u^{m}_h|^2_{1,h}\erwe\int_{t_m-\tau}^{t_{m+1}-\tau}\sum_{n=0}^{N-1}\chi_{n+1}(t,t+\tau)\,dt\leq \cter{K2}\tau
\end{align*}
by \cite[Lemma 6.2]{GHL} and Lemma \ref{remarkuhnrboundintomega}, where $\chi_{n+1}(t,t+\tau)=0$ for $t<0$.
Combining the previous estimates we arrive at
\begin{align}\label{210611_03}
\erwb\int_0^{T-\tau} A_1(t)\,dt\erwe\leq \left(\frac{\cter{K1}}{2}+2\cter{K2}\right)\tau.
\end{align}

Now we consider $A_2$. Applying Young's inequality we find
\begin{align*}
A_2(t)
&\leq\frac{\Delta t}{2}\sum_{n=0}^{N-1}\chi_{n+1}(t,t+\tau)\sum_{\sigma\in\edgesint}\frac{m_\sigma}{\dkl}(u_K^{n+1}-u_L^{n+1})^2\\
&\quad+\frac{\Delta t}{2}\sum_{n=0}^{N-1}\chi_{n+1}(t,t+\tau)\sum_{\sigma\in\edgesint}\frac{m_\sigma}{\dkl}\left|\int_{\nnt\Delta t}^{\net\Delta t} 
\left( g(u_K^{N,l})-g(u_L^{N,l}) \right) dW(s)\right|^2\\
&=:A_{2,1}(t)+A_{2,2}(t).
\end{align*}
We have
\begin{align*}
\erwb \int_0^{T-\tau}A_{2,1}(t) dt\erwe
&=\frac{\Delta t}{2}\erwb\sum_{n=0}^{N-1}|u_h^{n+1}|_{1,h}^2\int_0^{T-\tau}\chi_{n+1}(t,t+\tau) dt\erwe.
\end{align*}
By \eqref{chi} and Lemma \ref{remarkuhnrboundintomega} we may conclude
\begin{align}\label{estimatea11}
\erwb\int_0^{T-\tau} A_{2,1}(t)dt\erwe=\frac{\tau}{2}\erwb\int_0^T|\uhnr(s)|_{1,h}^2 \,ds\erwe\leq \frac{\cter{K1}\tau}{2}.
\end{align}
For the study of the term $A_{2,2}$ we recall the notation $u_K^{N,l}:=\uhnl(x_K)$ for $K\in\Tau$. From the Itô isometry it follows that for any $t\in (0,T-\tau)$ with $n_0(t)<n_1(t)$,
\begin{align*}
\erwb\left|\int_{\nnt\Delta t}^{\net\Delta t} 
\left( g(u_K^{N,l})-g(u_L^{N,l}) \right) \,dW(s)\right|^2 \erwe
&\leq \erwb\int_0^T|g(u_K^{N,l})-g(u_L^{N,l})|^2 \,ds\erwe.
\end{align*}
Therefore,
\begin{align*}
&\erwb\int_0^{T-\tau} A_{2,2}(t) dt\erwe\\
&=\frac{\Delta t}{2}\int_0^{T-\tau}\sum_{n=0}^{N-1}\chi_{n+1}(t,t+\tau)\sum_{\sigma\in\edgesint}\frac{m_\sigma}{\dkl}\erwb\left|\int_{\nnt\Delta t}^{\net\Delta t} \left(g(u_K^{N,l})-g(u_L^{N,l})\right) dW(s)\right|^2 \erwe dt\\
&\leq \frac{\Delta t}{2}\int_0^{T-\tau}\sum_{n=0}^{N-1}\chi_{n+1}(t,t+\tau)\sum_{\sigma\in\edgesint}\frac{m_\sigma}{\dkl}\erwb\int_0^T|g(u_K^{N,l})-g(u_L^{N,l})|^2 ds\erwe dt\\
&\leq L^2\frac{\Delta t}{2}\int_0^{T-\tau}\sum_{n=0}^{N-1}\chi_{n+1}(t,t+\tau)\sum_{\sigma\in\edgesint}\frac{m_\sigma}{\dkl}\erwb\int_0^T|u_K^{N,l}-u_L^{N,l}|^2 ds\erwe dt\\
&=L^2\frac{\Delta t}{2}\int_0^{T-\tau}\sum_{n=0}^{N-1}\chi_{n+1}(t,t+\tau)\,dt\int_0^T\erww{|\uhnl(s)|_{1,h}^2}\,ds.
\end{align*}
Because of \eqref{chi} there holds
\begin{align*}
\int_0^{T-\tau}\sum_{n=0}^{N-1}\chi_{n+1}(t,t+\tau)dt=\sum_{n=0}^{N-1}\tau=N\tau.
\end{align*}
Therefore from \eqref{210611_01} it follows
\begin{align}\label{estimatea12}
&\erwb\int_0^{T-\tau} A_{2,2}(t)dt\erwe
\leq \halbe L^2T\tau \cter{K2}.
\end{align}
Finally \eqref{zeittranslate} holds from \eqref{210611_03}, \eqref{estimatea11} and \eqref{estimatea12}.
\end{proof}

\begin{lem}\label{lemdefMhN}
There exists a constant $\ctel{K5}>0$, independent of the discretization parameters $N\in\mathbb{N}^{\ast}$ and $h$, such that for all $\tau \in (0,T)$ there holds
\begin{align}
\erww{\int_0^{T-\tau}\|\uhnl(t+\tau)-M_{h,N}(t+\tau)-(\uhnl(t)-M_{h,N}(t))\|_\lzlambda^2dt}\leq \cter{K5}\tau. \label{bmhn}
\end{align}
\end{lem}

\begin{proof}
Let $0<\tau<T$. We can write using the fact that for any $a,b,c\in \mathbb{R}$, $|a+b+c|^2\leqslant 3(|a|^2+|b|^2+|c|^2)$
\begin{align*}
&\erwb\int_0^{T-\tau}\|\uhnl(t+\tau)-M_{h,N}(t+\tau)-(\uhnl(t)-M_{h,N}(t))\|_\lzlambda^2 dt\erwe\\
&\leq 3\erwb\int_0^{T-\tau}\|\uhnl(t+\tau)-M_{h,N}^l(t+\tau)-(\uhnl(t)-M_{h,N}^l(t))\|_\lzlambda^2 dt\erwe\\
&\quad+ 3\erwb\int_0^{T-\tau}\|M_{h,N}(t+\tau)-M_{h,N}(t)\|_\lzlambda^2 dt\erwe\\
&\quad+ 3\erwb\int_0^{T-\tau}\|M_{h,N}^l(t+\tau)-M_{h,N}^l(t)\|_\lzlambda^2dt\erwe\\
&=: 3(I_1+I_2+I_3).
\end{align*}
From Lemma \ref{lemzeittranslate} we know that $I_1\leq \cter{K3}\tau$.

By the It\^{o} isometry, $H_2$, \eqref{H3} and Lemma \ref{uhnsupbound} we get
\begin{align*}
I_2
&=\int_0^{T-\tau}\int_t^{t+\tau}\erww{\|g(\uhnl(s))\|_\lzlambda^2}ds\,dt\\
&\leq C_L\int_0^{T-\tau}\int_t^{t+\tau}\left(|\Lambda|+\erww{\|\uhnl(s)\|_\lzlambda^2}\right)\,ds\,dt\\
& \leq C_L|\Lambda|T\tau+C_L\int_0^{T-\tau}\int_t^{t+\tau}\erww{\max_{n\in\{0,\ldots,N\}}\|u_h^n\|_\lzlambda^2}\,ds\,dt\\
&\leq C_LT(|\Lambda|+\cter{K4})\tau. 
\end{align*}
For $t\in[0,T]$, let $n_0(t),n_1(t)\in\{0,\dots,N-1\}$ be defined as in the proof of Lemma \ref{lemzeittranslate}.
From the It\^o isometry, $H_2$, \eqref{H3} and Lemma \ref{uhnsupbound} we get
\begin{align*}
I_3&= \int_0^{T-\tau}\int_{n_0(t)\Delta t}^{n_1(t)\Delta t}\erww{\|g(\uhnl(s))\|_\lzlambda^2}ds\,dt\\
&\leq C_L(|\Lambda|+\cter{K4})\int_0^{T-\tau}\int_{n_0(t)\Delta t}^{n_1(t)\Delta t} 1 \,ds\,dt.
\end{align*}
Similarly to the proof of Lemma \ref{lemzeittranslate}, let $\chi_{n}(t,t+\tau)=1$ for $n\in\na$ if $n\Delta t\in(t,t+\tau]$ and $0$ otherwise. Taking \eqref{chi} into account, we can continue the above estimate by
\begin{align*}
I_3&\leq C_L(|\Lambda|+\cter{K4})\int_0^{T-\tau}\sum_{n=0}^{N-1}\chi_{n+1}(t,t+\tau)\Delta t\,dt\\
&=C_L(|\Lambda|+\cter{K4})\Delta t\sum_{n=0}^{N-1}\int_0^{T-\tau}\chi_{n+1}(t,t+\tau)\,dt=C_LT(|\Lambda|+\cter{K4})\tau,
\end{align*}
and the assertion follows.
\end{proof}

\subsection{Bound on the Gagliardo seminorm}

In this subsection we give bounds on the approximate solutions which will be used in the stochastic compactness argument in Subsection \ref{S1}.
We denote by $[\ \cdot\ ]_{W^{\alpha,2}(\Lambda)}$ the Gagliardo seminorm such that for any function $w:\Lambda\to \re$ one has,
\begin{equation*}
 [\ w\ ]_{W^{\alpha,2}(\Lambda)}
 = \left(\int_\Lambda\int_\Lambda\frac{|w(x)-w(y)|^2}{|x-y|^{2+2\alpha}}dx\,dy\right)^\frac12.
\end{equation*}

Note that $W^{\alpha,2}(\Lambda)=\{w\in\lz:[\, w\, ]_{W^{\alpha,2}(\Lambda)}<\infty\}$.

\begin{lem}\label{spacialboundedness}
For any fixed $\alpha\in(0,\halbe)$, the sequences $(\uhnl)_{h,N}$ and  $(\uhnr)_{h,N}$ are bounded in $L^2(\Omega;L^2(0,T; W^{\alpha,2}(\Lambda)))$ independently of the discretization parameters $N\in\mathbb{N}^{\ast}$ and $h$.
\end{lem}

\begin{proof}
We fix $0<\alpha<\halbe$, $R>0$ and define $\buhnl$ as in Lemma \ref{210112_lem01}. For almost every $t\in(0,T)$, thanks to Lemma \ref{210112_lem01} 
\begin{align*}
&\int_{\re^2}\int_{\re^2}\frac{|\bar{u}_{h,N}^l(t,x)-\bar{u}_{h,N}^l(t,y)|^2}{|x-y|^{2+2\alpha}}dx\,dy\\
&=\int_{|\eta|>R}\int_{\re^2}\frac{|\bar{u}_{h,N}^l(t,x)-\bar{u}_{h,N}^l(t,x+\eta)|^2}{|\eta|^{2(1+\alpha)}}dx\,d\eta\\
&\quad+\int_{|\eta|<R}\int_{\re^2}\frac{|\bar{u}_{h,N}^l(t,x)-\bar{u}_{h,N}^l(t,x+\eta)|^2}{|\eta|^{2(1+\alpha)}}dx\,d\eta\\
&\leq 4\|\bar{u}_{h,N}^l(t)\|_{L^2(\re^2)}^2\int_{|\eta|>R}|\eta|^{-2(1+\alpha)}d\eta \\
&+ C\left(|\uhnl(t)|_{1,h}^2+\|\uhnl(t)\|_{L^2(\Lambda)}^2\right)
\int_{|\eta|<R}|\eta|^{-2(1+\alpha)+1}d\eta\\
&=4\|\uhnl(t)\|_{L^2(\Lambda)}^2\int_0^{2\pi}\int_R^\infty r^{-2(1+\alpha)}r\,drd\varphi\\
&+C
\left(|\uhnl(t)|_{1,h}^2+\|\uhnl(t)\|_{L^2(\Lambda)}^2\right)\int_0^{2\pi}\int_0^R r^{-2\alpha-1} r\,drd\varphi\\
&=: \tilde{C}_1\|\uhnl(t)\|_{L^2(\Lambda)}^2
+\tilde{C}_2|\uhnl(t)|_{1,h}^2
\end{align*}
where the constants $\tilde{C}_1,\tilde{C}_2\geq 0$ only depend on $\Lambda$ and $R>0$.
Consequently we have
\begin{align*}
\int_0^T[\uhnl(t)]_{W^{\alpha,2}(\Lambda)}^2 dt
&\leq\int_0^T\int_{\re^2}\int_{\re^2}\frac{|\bar{u}_{h,N}^l(t,x)-\buhnl(t,y)|^2}{|x-y|^{2+2\alpha}}dx\,dy\,dt\\
&\leq \int_0^T \left(\tilde{C}_1\|\uhnl(t)\|^2_{L^2(\Lambda)}+\tilde{C}_2|\uhnl(t)|_{1,h}^2 \right)dt.
\end{align*}
Therefore thanks to Lemmas \ref{remarkuhnrboundintomega} and~ \ref{uhnsupbound} we get
\begin{align*}
\erwb\|\uhnl\|_{L^2(0,T;W^{\alpha,2}(\Lambda))}\erwe&=\erwb\int_0^T\left(\|\uhnl(t)\|_{L^2(\Lambda)}+[\uhnl(t)]_{W^{\alpha,2}(\Lambda)}\right)^2 dt\erwe\\
&\leq 2(1+\tilde{C}_1)\int_0^T\erwb\|\uhnl(t)\|_{L^2(\Lambda)}^2\erwe\,dt
+2\tilde{C}_2 \int_0^T\erww{|\uhnl(t)|_{1,h}^2} dt\\
&\leq 2T\cter{K4}(1+\tilde{C}_1)+2\tilde{C}_2\cter{K2}.
\end{align*}
So $(\uhnl)_{h,N}$ is bounded in $L^2(\Omega;L^2(0,T;W^{\alpha,2}(\Lambda)))$ and also $(\uhnr)_{h,N}$ with similar arguments.
\end{proof}

In order to establish the $L^2(\Omega; W^{\alpha,2}(0,T;L^2(\Lambda)))$-bound on the discrete solutions we give the following auxiliary result:

\begin{lem}\label{uhnlMhNbound}
For any fixed $\alpha\in(0,\halbe)$, the sequence $(\uhnl-M_{h,N})_{h,N}$ defined by~\eqref{eq:defM} is bounded in $L^2(\Omega; W^{\alpha,2}(0,T;L^2(\Lambda)))$ independently of the discretization parameters $N\in\mathbb{N}^{\ast}$ and $h$.
\end{lem}

\begin{proof}
For any $x \in \Lambda$, let $\bar{\varphi}_{h,N}(t,x):=\uhnl(t,x)-M_{h,N}(t,x)$ for $t\in[0,T]$ and $\bar{\varphi}_{h,N}(t,x)=0$ for $t\in\re\setminus[0,T]$.
We have,
\begin{align*}
&\erwb\int_0^T\int_0^T\frac{\|\bar{\varphi}_{h,N}(s)-\bar{\varphi}_{h,N}(t)\|_{L^2(\Lambda)}^2}{|t-s|^{1+2\alpha}}\,ds\,dt\erwe\\
&=\erwb\int_0^T\int_0^t\frac{\|\bar{\varphi}_{h,N}(s)-\bar{\varphi}_{h,N}(t)\|_{L^2(\Lambda)}^2}{|t-s|^{1+2\alpha}}\,ds\,dt\erwe\\
&\quad+\erwb\int_0^T\int_t^T\frac{\|\bar{\varphi}_{h,N}(s)-\bar{\varphi}_{h,N}(t)\|_{L^2(\Lambda)}^2}{|t-s|^{1+2\alpha}}\,ds\,dt\erwe\\
&=\erwb\int_0^T\int_{0}^{t}\frac{\|\bar{\varphi}_{h,N}(t-\tau)-\bar{\varphi}_{h,N}(t)\|_{L^2(\Lambda)}^2}{|\tau|^{1+2\alpha}}\,d\tau\,dt\erwe\\
&\quad+\erwb\int_0^T\int_{0}^{T-t}\frac{\|\bar{\varphi}_{h,N}(t+\tau)-\bar{\varphi}_{h,N}(t)\|_{L^2(\Lambda)}^2}{|\tau|^{1+2\alpha}}\,d\tau\,dt\erwe
\end{align*}
Using Fubini's theorem we obtain
\begin{align*}
&\erwb\int_0^T\int_{0}^{t}\frac{\|\bar{\varphi}_{h,N}(t-\tau)-\bar{\varphi}_{h,N}(t)\|_{L^2(\Lambda)}^2}{|\tau|^{1+2\alpha}}\,d\tau\,dt\erwe\\
&\quad+\erwb\int_0^T\int_{0}^{T-t}\frac{\|\bar{\varphi}_{h,N}(t+\tau)-\bar{\varphi}_{h,N}(t)\|_{L^2(\Lambda)}^2}{|\tau|^{1+2\alpha}}\,d\tau\,dt\erwe\\
&=\erwb\int_0^T\int_{\tau}^{T}\frac{\|\bar{\varphi}_{h,N}(t-\tau)-\bar{\varphi}_{h,N}(t)\|^2_{L^2(\Lambda)}}{|\tau|^{1+2\alpha}}\,dt\,d\tau\erwe\\
&\quad+\erwb\int_0^T\int_{0}^{T-\tau}\frac{\|\bar{\varphi}_{h,N}(t+\tau)-\bar{\varphi}_{h,N}(t)\|^2_{L^2(\Lambda)}}{|\tau|^{1+2\alpha}}\,dt\,d\tau\erwe\\
&=\erwb\int_0^T\int_{0}^{T-\tau}\frac{\|\bar{\varphi}_{h,N}(s)-\bar{\varphi}_{h,N}(s+\tau)\|^2_{L^2(\Lambda)}}{|\tau|^{1+2\alpha}}\,ds\,d\tau\erwe\\
&\quad+\erwb\int_0^T\int_{0}^{T-\tau}\frac{\|\bar{\varphi}_{h,N}(t+\tau)-\bar{\varphi}_{h,N}(t)\|^2_{L^2(\Lambda)}}{|\tau|^{1+2\alpha}}\,dt\,d\tau\erwe\\
&= 2\int_0^T|\tau|^{-1-2\alpha}\int_0^{T-\tau}\erwb\|\bar{\varphi}_{h,N}(t+\tau)-\bar{\varphi}_{h,N}(t)\|_{L^2(\Lambda)}^2\erwe dt\,d\tau\\
&\leq 2\cter{K5}\int_0^T|\tau|^{-2\alpha}\,d\tau,
\end{align*}
where estimate~\eqref{bmhn} is used in the last inequality. Thus,  the above integral is finite for $\alpha\in\left(0,\halbe\right)$.

\end{proof}

\begin{lem}\label{timeboundedness}
For any fixed $\alpha\in(0,\halbe)$, the sequence $(\uhnl)_{h,N}$ is bounded in \linebreak $L^2(\Omega;W^{\alpha,2}(0,T;L^2(\Lambda)))$ independently of the discretization parameters $N\in\mathbb{N}^{\ast}$ and $h$.
\end{lem}

\begin{proof}
From Lemma \ref{uhnlMhNbound} we know that $(\uhnl-M_{h,N})_{h,N}$ is bounded in $L^2(\Omega;W^{\alpha,2}(0,T;\linebreak L^2(\Lambda)))$. By Lemma \ref{uhnsupbound},
\begin{align*}
\int_0^T\erww{\|\uhnl(t)\|_\lzlambda^2}dt\leq T\erww{\sup_{t\in[0,T]}\|\uhnl(t)\|_\lzlambda^2}\leq T\cter{K4}
\end{align*}
hence, by applying \cite[Lemma 2.1]{gatarekflandoli} it follows that $(M_{h,N})_{h,N}$ is bounded in\linebreak $L^2(\Omega;W^{\alpha,2}(0,T;L^2(\Lambda)))$. Now, since $\uhnl=(\uhnl-M_{h,N}^l)+M_{h,N}^l$, the assertion follows.
\end{proof}

\section{Convergence of the finite-volume scheme}\label{ConvFVscheme}

We now have all the necessary material to pass to the limit in the numerical scheme.

In the sequel, for $m\in\mathbb{N}^\ast$, let $(\Tau_m)_m$ be a sequence of admissible meshes of $\Lambda$ in the sense of Definition \ref{defmesh} such that the mesh size $h_m$ tends to $0$ when $m$ tends to $+\infty$ and let $(N_m)_m\subset\mathbb{N}$ be a sequence with $\lim_{m\rightarrow+\infty} N_m=+\infty$ and $\Delta t_m:=\frac{T}{N_m}$.\\ 
For the sake of simplicity we shall use the notations $\Tau=\Tau_m$, $h=\operatorname{size}(\Tau_m)$, $\Delta t=\Delta t_m$ and $N=N_m$ when the $m$-dependency is not useful for the understanding of the reader.

\subsection{Weak convergence of finite-volume approximations}
First, thanks to the bounds on the discrete solutions, we obtain the following weak convergences.

\begin{lem}\label{addreg u}
There exist not relabeled subsequences of $(u_{h,N}^r)_m$ and of $(u_{h,N}^l)_m$ respectively and a function $u\in L^2(\Omega;L^2( 0,T;H^1(\Lambda)))$ such that
\[u_{h,N}^l\rightharpoonup u \ \text{and} \ \uhnr\rightharpoonup u\]
for $m\rightarrow+\infty$ weakly in $L^2(\Omega;L^2(0,T;L^2(\Lambda)))$.
\end{lem}
\begin{proof}
From Lemma \ref{210611_lem01} it follows that the sequences $(u_{h,N}^r)_m$, $(\uhnl)_m$ respectively are bounded in $L^2(\Omega;L^2(0,T;L^2(\Lambda)))$, thus, up to a not relabeled subsequence, weakly convergent in $L^2(\Omega;L^2(0,T;L^2(\Lambda)))$ towards possibly distinct elements $u$, $\tilde{u}$ respectively. Moreover, from Lemma \ref{remarkuhnrboundintomega} and Remark \ref{remarkforuhnrboundiii}, it follows that 
\[\| \nabla^h \uhnr\|^2_{L^2(\Omega\times (0,T)\times \Lambda)}\leq 2 \cter{K1}.\]  
Consequently, there exists $\chi\in L^2(\Omega;L^2(0,T;L^2(\Lambda)))$ such that, passing to a not relabeled subsequence if necessary,
$\nabla^h \uhnr\rightharpoonup \chi$ weakly in $L^2(\Omega;L^2(0,T;L^2(\Lambda)))$ for $m\rightarrow+\infty$. With similar arguments as in \cite[Lemma 2]{eymardgallouet} and \cite[Theorem 14.3]{gal} we get the additional regularity $u\in L^2(\Omega;L^2( 0,T;H^1(\Lambda)))$ and $\chi=\nabla u$. Since, by Proposition~\ref{bounds},
\begin{align}\label{210824_05}
\erww{\|\uhnr-\uhnl\|_{L^2(0,T;\lzlambda)}^2}=\Delta t\erww{\sum_{n=0}^{N-1}\|u_h^{n+1}-u_h^n\|_\lzlambda^2}
\leq C_1\Delta t,
\end{align}
it follows that $(\uhnr-\uhnl)_m$ converges to $0$ strongly in $L^2(\Omega;L^2(0,T;L^2(\Lambda)))$ when $m\rightarrow+\infty$, hence also weakly and therefore $u=\tilde{u}$.
\end{proof}
Our aim is to show that $u$ is the unique solution to \eqref{equation}. But weak convergence is not enough to pass to the limit in the nonlinear diffusion term of our finite-volume scheme. Therefore we will apply the method of stochastic compactness.

\subsection{The stochastic compactness argument}\label{S1}
For better readability we define $V:=L^2(0,T;L^2(\Lambda))$ and
$$\mathcal{W}:=W^{\alpha,2}(0,T;L^2(\Lambda))\cap L^2(0,T;W^{\alpha,2}(\Lambda)).$$
From Lemmas \ref{spacialboundedness} and \ref{timeboundedness} we get immediately the following bound.
\begin{lem}\label{stochcomp}
For any fixed $\alpha\in(0,\halbe)$, there exists a constant $\ctel{K6}$ depending on $u_0$ and the mesh regularity $\reg$ but not depending on the discretization parameter $m\in\na^\ast$, such that
\begin{align*}
\erww{\|\uhnl\|_{\mathcal{W}}^2}\leq \cter{K6}.
\end{align*}
\end{lem}
In the following, for a random variable $X$ defined on a probability space $(\Omega,\mathcal{A},\mathds{P})$ the law of $X$ will be denoted by $\mathds{P}\circ X^{-1}$.

\begin{lem}\label{uhnltight}
The sequence of laws $(\mathds{P}\circ(\uhnl)^{-1})_m$ on $L^2(0,T;L^2(\Lambda))$ is tight.
\end{lem}

\begin{proof}
By \cite[Theorem 2.1]{gatarekflandoli} we know that $\mathcal{W}$ is compactly embedded in $V$. Let $\eps>0$ be arbitrary. For any $R>0$ the ball $B_\mathcal{W}(0,R):=\{v\in \mathcal{W}:\|v\|_\mathcal{W}\leq R\}$ is compact in $V$. There holds
\begin{equation*}
[\mathds{P}\circ(\uhnl)^{-1}](B_\mathcal{W}(0,R))=1-[\mathds{P}\circ(\uhnl)^{-1}](B_\mathcal{W}(0,R)^c)
=1-\int_{\{\|\uhnl\|_\mathcal{W}>R\}}1\,d\mathds{P}.
\end{equation*}
Then by using Markov inequality
\begin{equation*}
\int_{\{\|\uhnl\|_\mathcal{W}>R\}}1\, d\mathds{P}
\leq\frac{1}{R^2}\int_{\{\|\uhnl\|_\mathcal{W}>R\}}\|\uhnl\|_\mathcal{W}^2 \,d\mathds{P}
\leq \frac{1}{R^2}\erww{\|\uhnl\|_\mathcal{W}^2}
\leq\frac{\cter{K6}}{R^2}.
\end{equation*}
In the last inequality we used that $(\uhnl)_{h,N}$ is, thanks to Lemma~\ref{stochcomp}, bounded in $L^2(\Omega;\mathcal{W})$. It follows 
\begin{align*}
[\mathds{P}\circ(\uhnl)^{-1}](B_\mathcal{W}(0,R))\geq1-\frac{\cter{K6}}{R^2}.
\end{align*}
If we choose an appropriate $R$, the assertion follows.
\end{proof}

For the next lemmas, we recall that the initial value of Problem \eqref{equation} denoted $u_0$ is $\mathcal{F}_0$-measurable and belongs to $ L^2(\Omega;H^1(\Lambda))$. Moreover, its spatial discretization denoted $u_h^0$ is defined by \eqref{eq:def_u0}. In the following, we will write $(W(t))_{t\geq 0}=:W$, whenever the $t$-dependence is not relevant for the argumentation.

In order to apply Skorokhod theorem and to prove the almost sure convergence, we begin by proving the convergence in law.
\begin{lem}\label{convginlaw}
For $m\in\mathbb{N}^{\ast}$ we consider the sequence of random vectors 
\begin{align*}
Y_m=((\uhnlm,\uhnrm-\uhnlm,W,u_{h_m}^0)
\end{align*}
with values in 
\[\mathcal{X}:=L^2(0,T;\lzlambda)\times L^2(0,T;\lzlambda)\times C([0,T])\times L^2(\Lambda).\]
There exists a not relabeled subsequence of $(Y_m)_m$ converging in law, i.e., there exists a probability measure $\mu_{\infty}$ on $\mathcal{X}$ with marginal laws $\mu_{\infty}^1,\delta_0,\mathds{P}\circ W^{-1},\mathds{P}\circ (u_0)^{-1}$ such that
\[\erwb f(Y_m)\erwe\stackrel{m\rightarrow +\infty}{\longrightarrow}\int_{\mathcal{X}}f \,d\mu_{\infty}\]
for all bounded, continuous functions $f:\mathcal{X}\rightarrow\mathbb{R}$.
\end{lem}

\begin{proof}
We recall that a subsequence of $(Y_m)_m$ is tight if and only if all its components are tight. The tightness of laws of $(\uhnlm)_{m}$ was shown in Lemma \ref{uhnltight}. Then, from Prokhorov theorem (see \cite[Theorem 5.1]{Bil99}) it follows that, passing to a not relabeled subsequence if necessary, $(\uhnlm)_{m}$ converges in law towards a probability measure $\mu_{\infty}^1$ defined on $L^2(0,T;L^2(\Lambda))$.\\
Clearly, as a constant sequence, the Brownian motion $W$ converges in law towards $\mathds{P}\circ W^{-1}$. Since $(u_{h_m}^0)_m$ converges to $u_0$ in $L^2(\Lambda)$ for $m\rightarrow +\infty$ a.s. in $\Omega$ (see \cite[Proposition 3.5]{Andreianov}), it follows that $(u_{h_m}^0)_m$ converges in law towards $\mathds{P}\circ (u_0)^{-1}$.
From \eqref{210824_05} it follows that $(\uhnrm-\uhnlm)_m$ converges to $0$ for $m\rightarrow+\infty$ in $L^2(\Omega;L^2(0,T;L^2(\Lambda)))$ and this convergence implies for all bounded, continuous functions $f:L^2(0,T;L^2(\Lambda))\rightarrow\mathbb{R}$
\begin{align*}
\int_{L^2(0,T;L^2(\Lambda))}f \,d(\mathds{P}\circ(\uhnrm-\uhnlm)^{-1})=\erww{f(\uhnrm-\uhnlm)}\stackrel{m\rightarrow+\infty}{\longrightarrow}\erww{f(0)},
\end{align*}
hence the convergence in law of $(\uhnrm-\uhnlm)_m$ towards $\delta_0$.
\end{proof}

Thanks to Lemma \ref{convginlaw} we may apply the Skorokhod representation Theorem (see \cite[Theorem 6.7]{Bil99}): There exists a probability space $(\Omega',\mathcal{A}',\mathds{P}')$ and random variables
\[Y'_m = (v_m,z_m,B_m,v^0_m), \ u_{\infty}, W_{\infty}, v_0\]
with
\begin{align*}
&\mathds{P}'\circ(Y'_m)^{-1} = \mathds{P}\circ(Y_m)^{-1}  \ \text{for all} \ m\in\na,\\
&\mathds{P}'\circ(v_0)^{-1}=\mathds{P}\circ(u_0)^{-1},\\
&\mathds{P}'\circ (u_{\infty})^{-1}=\mu_{\infty}^1,\\
&\mathds{P}'\circ (W_{\infty})^{-1}=\mathds{P}\circ W^{-1}
\end{align*}
and such that
\begin{align}
\label{eq:conv}
\begin{aligned}
v_m &\stackrel{m\rightarrow+\infty}{\longrightarrow} u_\infty &&\text{ in }L^2(0,T;\lzlambda),\ \mathds{P'}\text{-a.s. in }\Omega'\\
z_m&\stackrel{m\rightarrow+\infty}{\longrightarrow} 0 &&\text{ in }L^2(0,T;\lzlambda),\ \mathds{P'}\text{-a.s. in }\Omega'\\
B_m&\stackrel{m\rightarrow+\infty}{\longrightarrow} W_\infty &&\text{ in }C([0,T]),\ \mathds{P'}\text{-a.s. in }\Omega'\\
v^0_m &\stackrel{m\rightarrow+\infty}{\longrightarrow} v_0&&\text{ in }L^2(\Lambda),\ \mathds{P'}\text{-a.s. in }\Omega'.
\end{aligned}
\end{align}
In Lemmas \ref{lem:id_vm} and \ref{220111_lem01} we will show that, thanks to equality in law, $v_m$ and $z_m$ are in fact finite-volume functions with the same piecewise constant structure as $\uhnlm$ and $\uhnrm-\uhnlm$, respectively.

\begin{lem}
\label{lem:id_vm}
For $m\in\mathbb{N}^{\ast}$ fixed, $v_m$ is a step function with respect to time and space in the sense that there exists $v_{h_m,N_m}^l\in \mathbb{R}^{d_{h_m}\times {N_m}}$ such that $\mathds{P}'$-a.s. in $\Omega'$ 
$v_m=v_{h_m,N_m}^l$.
Moreover, $v_{h_m,N_m}^l(0,x):=v_{h_m}^0(x)=v^0_m(x)$ for all $x\in\Lambda$
and, in particular $v^0_m=v_{h_m}^0$ is a spatial step function.
\end{lem}

\begin{proof}
By \cite[Lemma A3]{VZ21} with $E=L^2(0,T;L^2(\Lambda))$ and $F=\mathbb{R}^{d_{h_m}\times N_m}$ it follows that there exists $(v_K^n)_{\substack{K\in\Tau_m \\ n\in\{0,\ldots,N_m-1\}}}$ in $\mathbb{R}^{d_{h_m}\times{N_m}}$ such that
\begin{align*}
v_m &\equiv(v_K^n)_{\substack{K\in\Tau_m \\ n\in\{0,\ldots,N_m-1\}}} \text{$\mathds{P}'-$a.s. in} \ \Omega'.
\end{align*}
In the same manner with $E=L^2(\Lambda)$ and $F=\mathbb{R}^{d_{h_m}}$ it follows that there exists $(\tilde{v}_K^0)_{K\in\Tau_m}$ in $\mathbb{R}^{d_{h_m}}$ such that
\begin{align*}
v_m^0&\equiv (\tilde{v}_K^0)_{K\in\Tau_m} \ \text{$\mathds{P}'$-a.s. in} \ \Omega'.
\end{align*}
We recall the notation of Subsection \ref{discretenotation} and in particular that 
\begin{align*}
\uhnlm &\equiv (u_K^n)_{\substack{K\in\Tau_m \\ n\in\{0,\ldots,N_m-1\}}}\text{$\mathds{P}-$a.s. in} \ \Omega.
\end{align*} 
For any $K\in\Tau_m$ we consider the non-negative, Borel measurable mapping 
\begin{align*}
\xi_K^0:\mathbb{R}^{d_{h_m}}\times\mathbb{R}^{d_{h_m}\times N_m}&\rightarrow\mathbb{R} \\ 
((a_M)_M, (b_M^k)_{M,k})&\mapsto |a_K-b_K^0|.
\end{align*}
Since 
\begin{align*}
\mathds{P}\circ ((u_M^{0})_{M}, (u_M^{k})_{M,k}))^{-1}=\mathds{P}'\circ ((\tilde{v}_M^0)_{M}, (v_M^k)_{M,k})^{-1},
\end{align*}
it follows that
\begin{align*}
0=\erww{\xi_K^0((u_M^{0})_{M}, (u_M^{k})_{M,k})}=\erws{\xi_K^0((\tilde{v}_M^0)_{M}, (v_M^k)_{M,k})}=\erws{|\tilde{v}_K^0-v_K^0|}
\end{align*}
and therefore, for all $x\in K$ and all $K\in\Tau_m$, $v_m(0,x)=v_K^0=\tilde{v}_K^0=v_m^0(x)$ $\mathds{P}'$-a.s. in $\Omega'$.
\end{proof}

\begin{lem}\label{220111_lem01}
For $m\in\mathbb{N}^{\ast}$ fixed, $z_m(t,x)=v^{n+1}_K-v_K^n$ for all $(t,x)\in (t_n,t_{n+1}]\times K$ and $\mathds{P}'$-a.s. in $\Omega'$ for any $K\in\Tau_m$ and $n\in\{0,\ldots,N_m-1\}$, where $(v_K^n)_{\substack{K\in\Tau_m \\ n\in\{0,\ldots,N_m-1\}}}$ is defined as in the proof of Lemma \ref{lem:id_vm}.
\end{lem}

\begin{proof}
With similar arguments as in Lemma \ref{lem:id_vm} it follows that there exists \linebreak $(z_K^{n})_{\substack{K\in\Tau_m \\ n\in\{0,\ldots,N_m-1\}}}\in\mathbb{R}^{d_{h_m}\times{N_m}}$ such that
\begin{align*}
z_m\equiv(z_K^{n})_{\substack{K\in\Tau_m \\ n\in\{0,\ldots,N_m-1\}}} \ \text{$\mathds{P}'-$a.s. in} \ \Omega'.
\end{align*}
For any fixed $K\in\Tau_m$, $n\in \{0,\ldots,N_m-1\}$, the mapping
\begin{align*}
\Phi_K^n:\mathbb{R}^{d_{h_m}\times {N_m}}\times\mathbb{R}^{d_{h_m}\times N_m}\rightarrow\mathbb{R}, \ 
((a_M^k)_{M,k}, (b_M^k)_{M,k})\mapsto|a_K^{n+1}-a_K^n-b_K^n|
\end{align*}
is non-negative and Borel measurable. Since
\[\mathds{P}\circ ((u_M^k)_{M,k}, (u_M^{k+1}-u_M^k)_{M,k})^{-1}=\mathds{P}'\circ ((v_M^k)_{M,k}, (z_M^k)_{M,k})^{-1},\]
it follows that for any $K\in\Tau_m$ and all $n\in\{0,\ldots,N_m-1\}$
\begin{align*}
0&=\erwb\Phi_K^n((u_M^k)_{M,k}, (u_M^{k+1}-u_M^k)_{M,k})\erwe=\erws{\Phi_K^n((v_M^k)_{M,k}, (z_M^k)_{M,k})}\\
&=\erws{|v_K^{n+1}-v_K^n-z_K^n|}.
\end{align*}
Therefore, for all $K\in\Tau_m$ and all $n\in\{0,\ldots,N_m-1\}$ there holds $z_K^n=v_K^{n+1}-v_K^n$ $\mathds{P}'$-a.s. in $\Omega'$.
\end{proof}

Next we prove that the finite-volume function $(v_h^n)_{1\leq n\leq N}$ we have just constructed verifies the following numerical scheme.

\begin{lem}
For $m\in\mathbb{N}^{\ast}$ fixed, any $n\in\{0,\dots,N_m-1\}$ and any $K\in\Tau_m$, $v_K^{n+1}$ satisfies the semi-implicit equation
\begin{align}\label{newsemiimplicitequ}
\frac{m_K}{\Delta t}(v_K^{n+1}-v_K^n)+\sum_{\sigma\in\edgesint\cap\edges_K}\frac{m_\sigma}{\dkl}(v_K^{n+1}-v_L^{n+1})-\frac{m_K}{\Delta t}g(v_K^n)\Delta_{n+1}B_m=0
\end{align}
$\mathds{P}'$-a.s. in $\Omega'$, where $\Delta_{n+1}B_m:=B_m(t_{n+1})-B_m(t_n)$.
\end{lem}

\begin{proof}
From Lemma \ref{220111_lem01} it follows that $z_K^n=v_K^{n+1}-v_K^n$, $\mathds{P}'$-a.s. in $\Omega'$ for all $K\in\Tau_m$ and all $n\in\{0,\ldots,N_m-1\}$.
Then, for arbitrary $K\in\Tau_m$, $n\in\{0,\dots,N_m-1\}$ the mapping
\begin{align*}
\Psi_K^n:\mathbb{R}^{d_{h_m}\times N_m}\times\mathbb{R}^{d_{h_m}\times N_m} \times C([0,T])&\rightarrow\mathbb{R}\\
((a_M^k)_{M,k}, (b_M^k)_{M,k},f)&\mapsto\bigg|\frac{m_K}{\Delta t}b_K^n+\hspace*{-0.5cm}\sum_{\sigma\in\edgesint\cap\edges_K}\frac{m_\sigma}{\dkl}(b_K^n+a_K^n)-(b_L^n+a_L^n)\\
&\qquad-\frac{m_K}{\Delta t}g(a_K^n)(f(t_{n+1})-f(t_n))\bigg|
\end{align*}
is non-negative and Borel measurable. Since
\[\mathds{P}\circ ((u_M^k)_{M,k}, (u_M^{k+1}-u_M^k)_{M,k}, W)^{-1}=\mathds{P}'\circ ((v_M^k)_{M,k}, (z_M^k)_{M,k}, B_m)^{-1},\]
from Proposition \ref{210609_prop1} it follows that
\begin{align*}
0&=\erww{\Psi_K^n((u_M^k)_{M,k}, (u_M^{k+1}-u_M^k)_{M,k}, W)}=\erws{\Psi_K^n ((v_M^k)_{M,k}, (z_M^k)_{M,k}, B_m)}\\
&=\erws{\left|\frac{m_K}{\Delta t}(v_K^{n+1}-v_K^n)+\sum_{\sigma\in\edgesint\cap\edges_K}\frac{m_\sigma}{\dkl}(v_K^{n+1}-v_L^{n+1})-\frac{m_K}{\Delta t}g(v_K^n)\Delta_{n+1}B_m\right|}.
\end{align*}
Thus, for all $K\in\Tau_m$, $n\in\{0,\dots,N_m-1\}$ and $\mathds{P}'$-a.s. in $\Omega'$
\begin{align*}
0=\frac{m_K}{\Delta t}(v_K^{n+1}-v_K^n)+\sum_{\sigma\in\edgesint\cap\edges_K}\frac{m_\sigma}{\dkl}(v_K^{n+1}-v_L^{n+1})-\frac{m_K}{\Delta t}g(v_K^n)\Delta_{n+1}B_m.
\end{align*}
\end{proof}

\subsection{Identification of the stochastic integral}\label{S2}
In this subsection, we adapt ideas taken from \cite{Debussche}, \cite{BreitFeireisl}, \cite{OPW20} and adjust the arguments to our specific situation. We show that each $B_m$ is a Brownian motion with respect to the filtration given in Definition \ref{filtration2}. With this result at hand, we may show that $(W_{\infty}(t))_{t\geq 0}$ is a Brownian motion with respect to the filtration given in Definition \ref{filtration}. In Lemma \ref{measurability} we prove that $u_{\infty}$ is admissible for the stochastic It\^{o} integral with respect to $(W_{\infty}(t))_{t\geq 0}$. Finally, in Lemma \ref{martingaleIdentification} we provide an approximation result for the stochastic It\^{o} integrals.

\begin{defi}\label{filtration2} 
For $t\in [0,T]$ we define $\mathcal{F}_t^{m}$ to be the smallest sub-$\sigma$-field of $\mathcal{A}'$ generated by $v_m^0$ and $B_m(s)$ for $0\leq s\leq t$. The right-continuous, $\mathds{P}'$-augmented filtration of $(\mathcal{F}_t^{m})_{t\in [0,T]}$ denoted by $(\mathfrak{F}_t^m)_{t\in[0,T]}$ is defined by
\[ \mathfrak{F}_t^m:=\bigcap_{s>t}\sigma\left[\mathcal{F}_s^m\cup \{\mathcal{N}\in\mathcal{A}':\mathds{P}'(\mathcal{N})=0\}\right]\]
for any $t\in [0,T]$.
\end{defi}

\begin{remark}
We recall that for the augmented filtration and for given processes $(X_t)_{t\geq0}$, $(Y_t)_{t\geq0}$ such that $(X_t)_{t\geq0}$ is adapted and $Y_t=X_t$ holds a.s. for all $t$, it holds true that $(Y_t)_{t\geq0}$ is also adapted (see, e.g., \cite[p.35]{Baldi}).
\end{remark}

\begin{lem}\label{lemWNBM}
$(v_m)_m$ is adapted to $(\mathfrak{F}_t^m)_{t\in[0,T]}$ and $(B_{m}(t))_{t\in[0,T]}$ is a Brownian motion with respect to $(\mathfrak{F}_t^m)_{t\in[0,T]}$.
\end{lem}

\begin{proof}
Since $(\mathfrak{F}_t^m)_{t\in[0,T]}$ is a filtration induced by $v_m^0$ and $B_{m}$, in particular $v_m^0$ is $\mathfrak{F}^m_0$-measurable. Thus, applying the same arguments as in the proof of Proposition \ref{210609_prop1}, from \eqref{newsemiimplicitequ} it follows that $v_m$ is adapted to $(\mathfrak{F}_t^m)_{t\in[0,T]}$.
Since $\mathds{P}'\circ (B_{m})^{-1}=\mathds{P}\circ W^{-1}$, we get the following results:
\begin{itemize}
\item $\erws{|B_{m}(0)|}=\erwb|W(0)|\erwe=0$, hence $B_{m}(0)=0$ $\mathds{P}'$-a.s. in $\Omega'$.
\item By Burkholder-Davis-Gundy inequality there exists a constant $C_B>0$ such that
\begin{align}\label{WNsupbound}
\erws{\sup_{t\in[0,T]}|B_{m}(t)|^2}=\erww{\sup_{t\in[0,T]}|W(t)|^2}\leq C_B T^\halbe<\infty.
\end{align}
\item For all $0\leq s\leq t\leq T$ and all bounded, continuous functions $\psi:C_b(L^2(\Lambda)\times C([0,s]))\rightarrow\mathbb{R}$
\begin{align}\label{211001_04}
0=\erww{(W(t)-W(s))\psi(u_{h_m}^0,W|_{[0,s]})}
=\erws{(B_{m}(t)-B_{m}(s))\psi(v_m^0,B_{m}|_{[0,s]})},
\end{align}
and
\begin{align}\label{211001_05}
\begin{aligned}
0&=\erww{(W^2(t)-W^2(s)-(t-s))\psi(u_{h_m}^0,W|_{[0,s]})}\\
&=\erws{(B_{m}^2(t)-B_{m}^2(s)-(t-s))\psi(v_{m}^0,B_{m}|_{[0,s]})}
\end{aligned}
\end{align}
\end{itemize}
Recalling Definition \ref{filtration2}, $\mathcal{F}_t^m=\sigma_t(v_m^0,B_{m})$ for $t\in [0,T]$. The real-valued random variable 
\[\Omega'\ni\omega'\mapsto\psi(v_m^0(\omega'), B_{m}|_{[0,s]}(\omega'))\]
is $\mathcal{F}_s^m$-measurable. Using the properties of conditional expectation from \eqref{211001_04} it follows that
\begin{align}\label{211001_06}
\begin{aligned}
0=&\erws{(B_{m}(t)-B_{m}(s))\psi(v_m^0,B_{m}|_{[0,s]})}\\
=&\erws{\mathbb{E}'\left((B_{m}(t)-B_{m}(s))\psi(v_m^0,B_{m}|_{[0,s]})|\,\mathcal{F}^m_s\right)}\\
=&\erws{\psi(v_m^0,B_{m}|_{[0,s]})\mathbb{E}'\left(B_{m}(t)-B_{m}(s)|\,\mathcal{F}^m_s\right)}.
\end{aligned}
\end{align}
Since \eqref{211001_06} applies to every bounded and continuous function $\psi:C_b(L^2(\Lambda)\times C([0,s]))\rightarrow\mathbb{R}$, from the Lemma of Doob-Dynkin (see, e.g., \cite[Proposition 3]{RS06}) it follows that
\begin{align*}
0=\erws{\mathds{1}_A\mathbb{E}'\left(B_{m}(t)-B_{m}(s)|\,\mathcal{F}^m_s\right)}
\end{align*}
for all $\mathcal{F}^m_s$-measurable subsets $A\in \mathcal{A}'$ and for all $0\leq s\leq t\leq T$.  From the above equation it now follows that $\mathbb{E}'\left(B_{m}(t)-B_{m}(s)|\,\mathcal{F}^m_s\right)=0$ $\mathds{P}'$-a.s. in $\Omega'$ for all $0\leq s\leq t\leq T$ and therefore $(B_{m}(t))_{t\in [0,T]}$ is a martingale with respect to $(\mathcal{F}^m_t)_{t\in [0,T]}$.
Using \cite[p.75]{DM80} we may conclude that $(B_{m}(t))_{t\in [0,T]}$ is also a martingale with respect to the augmented filtration $(\mathfrak{F}^m_t)_{t\in[0,T]}$.
With similar arguments from \eqref{211001_05} it follows that $((B_{m}(t))^2-t)_{t\in [0,T]}$ is a martingale with respect to $(\mathfrak{F}^m_t)_{t\in[0,T]}$ and consequently the quadratic variation process $\langle\langle B_{m}\rangle\rangle_t$ of $(B_{m}(t))_{t\in[0,T]}$ is given by $t$ for all $t\in[0,T]$ (for the Definition of the quadratic variation of a stochastic process see \cite[Definition 2.19]{Prohl}).
Summarizing the above results, $(B_{m}(t))_{t\in[0,T]}$ is a square integrable martingale with respect to $(\mathfrak{F}^m_t)_{t\in[0,T]}$ starting in $0$ with almost surely continuous paths and quadratic variation $\langle\langle B_m\rangle\rangle_t=t$. From \cite[Theorem 3.11]{DPZ14} $(B_{m}(t))_{t\in[0,T]}$ is a Brownian motion with respect to $(\mathfrak{F}_t^m)_{t\in[0,T]}$.
\end{proof}

In the following, we want to show firstly that the stochastic process $(W_{\infty}(t))_{t\in [0,T]}:=W_{\infty}$ is a Brownian motion and secondly that a filtration may be chosen in order to have compatibility of $u_{\infty}$ with stochastic integration in the sense of It\^{o} with respect to $W_{\infty}$. Since $u_{\infty}$ is a random variable taking values in $L^2(0,T;L^2(\Lambda))$, $u_{\infty}(t,\cdot)$ is only defined for a.e. $t\in [0,T]$ and the construction of an appropriate filtration induced by $u_{\infty}$ becomes delicate. 

\begin{defi}\label{filtration}
For $t\in [0,T]$ let $\mathcal{F}_t^{\infty}$ be the smallest sub-$\sigma$-field of $\mathcal{A}'$ generated by $v_0$, $W_{\infty}(s)$ and $\int_0^s u_{\infty}(r)\,dr$ for $0\leq s\leq t$. 
The right-continuous, $\mathds{P}'$-augmented filtration of $(\mathcal{F}_t^{\infty})_{t\in [0,T]}$ denoted by $(\mathfrak{F}_t^\infty)_{t\in[0,T]}$ is defined by
\[\mathfrak{F}_t^\infty:=\bigcap_{s>t}\sigma\left[\mathcal{F}_s^{\infty}\cup\{\mathcal{N}\in\mathcal{A}': \mathds{P}'(\mathcal{N})=0\}\right]\]
for $t\in [0,T]$.
\end{defi}

In the following, we will show that $W_{\infty}$ is a Brownian motion with respect to $(\mathfrak{F}_t^\infty)_{t\in[0,T]}$ and $u_{\infty}$ admits a $(\mathfrak{F}_t^\infty)_{t\in[0,T]}$-predictable representative.

\begin{lem}\label{210830_lem1}
There holds $B_{m}\stackrel{m\rightarrow+\infty}{\longrightarrow} W_\infty$ in $L^2(\Omega';C([0,T]))$ and $(W_\infty(t))_{t\in[0,T]}$ is a Brownian motion with respect to $(\mathfrak{F}_t^\infty)_{t\in[0,T]}$.
\end{lem}

\begin{proof}
Combining \eqref{WNsupbound} with $\mathds{P}'\circ (W_{\infty})^{-1}=\mathds{P}\circ W^{-1}$ it follows that
\begin{align*}
\erws{\sup_{t\in[0,T]}|W_{\infty}(t)|^2}=\erww{\sup_{t\in[0,T]}|W(t)|^2}<\infty
\end{align*}
and consequently $W_{\infty}\in L^2(\Omega';C([0,T]))$. Moreover, since $\mathds{P}'\circ B_{m}^{-1}=\mathds{P}\circ W^{-1}$, it follows that
\begin{align*}
\erws{\sup_{t\in[0,T]}|B_{m}(t)|^2}=\erws{\sup_{t\in[0,T]}|W_{\infty}(t)|^2},\quad\forall m\in\na^\ast.
\end{align*}
We already know, that, for $m\rightarrow+\infty$, $B_{m}$ converges to $W_\infty$ in $C([0,T])$ a.s. in $\Omega'$. Therefore, a version of the Lemma of Brézis and Lieb (see \cite[Lemma A2]{VZ21}) provides the desired convergence result in $L^2(\Omega';C([0,T]))$. From
\[\mathds{P}'\circ (v_m^0,B_{m}, v_m)^{-1}=\mathds{P}\circ(u_{h_m}^0,W,u_{h_m,N_m}^l)^{-1}\]
it follows that for any $0\leq s\leq t\leq T$ and every bounded and continuous function $\psi:L^2(\Lambda)\times C([0,s])\times C([0,s];L^2(\Lambda))\rightarrow\mathbb{R}$, we have
\begin{align}\label{question7}
\begin{aligned}
&\erws{(B_{m}(t)-B_{m}(s))\psi\left(v_m^0, B_{m}|_{[0,s]}, \int_0^{\cdot} v_m(r)\,dr|_{[0,s]}\right)}\\
&=\erww{(W(t)-W(s))\psi\left(u_{h_m}^0, W|_{[0,s]}, \int_0^{\cdot} u_{h_m,N_m}^l(r)\,dr|_{[0,s]}\right)}.
\end{aligned}
\end{align}
Now, using the fact that $u_{h_m}^0$ is $\mathcal{F}_0$-measurable, by construction $\int_0^s \uhnlm(r)\,dr$ is $\mathcal{F}_s$-measurable for all $m\in\mathbb{N}$ and that $(W(t))_{t\geq 0}$ is a martingale with respect to $(\mathcal{F}_t)_{t\geq 0}$ one gets that
\begin{align}\label{211001_01}
\erww{(W(t)-W(s))\psi\left(u_{h_m}^0, W|_{[0,s]}, \int_0^{\cdot} u_{h_m,N_m}^l(r)\,dr|_{[0,s]}\right)}=0.
\end{align}
We recall that, $\mathds{P}'$-a.s. in $\Omega'$, $v_m^0\rightarrow v_0$ in $L^2(\Lambda)$ and $B_{m}\rightarrow W_{\infty}$ in $C([0,T])$, hence also in $C([0,s])$ for all $0\leq s\leq t\leq T$. Moreover, by Cauchy-Schwarz inequality,
\begin{align*}
&\left\| \int_0^{\cdot} v_m(r)\,dr-\int_0^{\cdot} u_{\infty}(r)\,dr\right\|_{C([0,s];L^2(\Lambda))}^2
=\sup_{z\in [0,s]}\left\|\int_0^z \left(v_m(r)-u_{\infty}(r)\right)\,dr\right\|_{L^2(\Lambda)}^2\\
&\leq \sup_{z\in [0,s]}\left(\int_0^z \| v_m(r)-u_{\infty}(r)\|_{L^2(\Lambda)}\,dr\right)^2\leq T\int_0^T\| v_m(r)-u_{\infty}(r)\|_{L^2(\Lambda)}^2\,dr.
\end{align*}
Since $v_m\stackrel{m\rightarrow+\infty}{\longrightarrow} u_{\infty}$ in $L^2(0,T;L^2(\Lambda))$ $\mathds{P}'$-a.s. in $\Omega'$, it follows that
\[\int_0^{\cdot}v_m(r)\,dr\stackrel{m\rightarrow+\infty}{\longrightarrow}\int_0^{\cdot} u_{\infty}(r)\,dr\]
in $C([0,s];L^2(\Lambda))$, $\mathds{P}'$-a.s. in $\Omega'$. Using the convergence of $B_{m}$ towards $W_{\infty}$ in $L^2(\Omega';C([0,T]))$, the convergence results from above and Lebesgue's dominated convergence theorem it follows that
\begin{align}\label{211001_02}
\begin{aligned}
&\lim_{m\rightarrow\infty}\erws{(B_{m}(t)-B_{m}(s))\psi\left(v_m^0, B_{m}|_{[0,s]}, \int_0^{\cdot} v_m(r)\,dr|_{[0,s]}\right)}\\
&=\erws{(W_{\infty}(t)-W_{\infty}(s))\psi\left(v_0, W_{\infty}|_{[0,s]}, \int_0^{\cdot} u_{\infty}(r)\,dr|_{[0,s]}\right)}.
\end{aligned}
\end{align}
Now, combining \eqref{question7}, \eqref{211001_01} and \eqref{211001_02} we obtain
\begin{align}\label{211001_07}
\erws{(W_{\infty}(t)-W_{\infty}(s))\psi\left(v_0, W_{\infty}|_{[0,s]}, \int_0^{\cdot} u_{\infty}(r)\,dr|_{[0,s]}\right)}=0.
\end{align}
From \eqref{211001_07} it follows that
\begin{align*}
\mathbb{E}'\left(W_{\infty}(t)-W_{\infty}(s)|\,\mathcal{F}^{\infty}_s\right)=0
\end{align*}
$\mathds{P}'$-a.s. in $\Omega'$ for all $0\leq s\leq t\leq T$.
With similar arguments as used for equation \eqref{211001_07}, we also get
\begin{align*}
&\erws{(W^2_{\infty}(t)-W^2_{\infty}(s)-(t-s))\psi\left(v_0, W_{\infty}|_{[0,s]}, \int_0^{\cdot} u_{\infty}(r)\,dr|_{[0,s]}\right)}=0.
\end{align*}
Now, using a similar argumentation as in the end of the proof of Lemma \ref{lemWNBM}, it follows that $(W_{\infty}(t))_{t\in[0,T]}$ is a Brownian motion with respect to $(\mathfrak{F}_t^{\infty})_{t\in[0,T]}$.
\end{proof}

By \cite[Theorem 2.6.3]{BreitFeireisl} it is always possible to choose $(\Omega',\mathcal{A}',\mathds{P}')=([0,1],\mathcal{B}([0,1]),\lambda)$, where $\mathcal{B}([0,1])$ denotes the Borel sets on $[0,1]$ and $\lambda$ denotes the Lebesgue measure on $[0,1]$. We will need this particular choice of the new probability space in the proof of the following Lemma.

We recall that, for a filtered probability space $(\Omega,\mathcal{A}, \mathds{P})$ with $\mathcal{F}=(\mathcal{F}_t)_{t\geq 0}$ and $T>0$, the predictable $\sigma$-field on $\Omega\times [0,T]$ is the $\sigma$-field generated by the sets 
\[(s,t]\times F_s, \ 0\leq s<t\leq T, \ F_s\in \mathcal{F}_s \ \text{and} \ \{0\}\times F_0, \ F_0\in\mathcal{F}_0.\]
For more details on stochastic integration in infinite dimension, we refer to \cite{DPZ14}.

\begin{lem}\label{measurability}
There exists a $(\mathfrak{F}_t^\infty)_{t\in[0,T]}$-predictable, $d\mathds{P}'\otimes dt$-representative of $u_\infty$.
\end{lem}

\begin{proof}
For $\delta>0$ we define $u_\infty^\delta:\Omega'\times[0,T]\rightarrow L^2(\Lambda)$ by
\begin{align*}
u_{\infty}^{\delta}(t):=\frac{1}{\delta}\int_{(t-\delta)^+}^{t}u_\infty(s)\,ds=\frac{1}{\delta}\left(\int_0^t u_\infty(s)\,ds-\int_0^{(t-\delta)^+}u_\infty(s)\,ds\right)
\end{align*}
where the integrals on the right-hand side are understood as Bochner integrals with values in $L^2(\Lambda)$. Since $u_{\infty}^{\delta}$ is an $(\mathfrak{F}_t^\infty)_{t\in[0,T]}$-adapted stochastic process with a.s. continuous paths, it is predictable with respect to $(\mathfrak{F}^{\infty}_t)_{t\in[0,T]}$.
For fixed $k\in\na$ the cut-off function $T_k:\re\rightarrow [-k,k]$ defined by $T_k(r):=r$ if $|r|<k$ and $T_k(r):=\sign(r)k$ if $|r|\geq k$ induces a continuous operator $L^2(\Lambda)\ni v\mapsto T_k(v)\in L^2(\Lambda)$. Hence the stochastic process
\begin{align*}
\Omega'\times[0,T]\ni(\omega',t)\mapsto T_k(u_\infty^\delta(\omega',t))\in L^2(\Lambda)
\end{align*}
is $(\mathfrak{F}_t^\infty)_{t\in[0,T]}$-predictable. Again we recall that, $\mathds{P}'$-a.s. in $\Omega'$, $v_m\stackrel{m\rightarrow+\infty}{\longrightarrow} u_{\infty}$ in $L^2(0,T;L^2(\Lambda))$, hence also in $L^1(0,T;L^2(\Lambda))$ and therefore
\[\lim_{m\rightarrow+\infty}\int_0^T \| v_m(t)\|_{L^2(\Lambda)}\,dt=\int_0^T \| u_{\infty}(t)\|_{L^2(\Lambda)}\,dt\]
$\mathds{P}'$-a.s. in $\Omega'$. Using Fatou's lemma, $\mathds{P}'\circ(v_m)^{-1}=\mathds{P}\circ(\uhnlm)^{-1}$, and the Cauchy-Schwarz inequality we obtain
\begin{multline*}
\erws{\int_0^T\| u_{\infty}(t)\|_{L^2(\Lambda)}\,dt}\leq\liminf_{m\rightarrow\infty}\erws{\int_0^T \| v_m(t)\|_{L^2(\Lambda)}\,dt}\\
=\liminf_{m\rightarrow\infty}\erww{\int_0^T \| \uhnlm(t)\|_{L^2(\Lambda)}\,dt}\leq \sqrt{T}\liminf_{m\rightarrow\infty}\| \uhnlm\|^2_{L^2(\Omega;L^2(0,T;L^2(\Lambda)))}.
\end{multline*}
From Lemma \ref{210611_lem01} it follows that the right-hand side of the equation is uniformly bounded and consequently, $u_{\infty}\in L^1(\Omega';L^1(0,T;L^2(\Lambda)))$.
In particular, $u_{\infty}\in L^1(\Omega';L^1(0,T;L^1(\Lambda)))$. Since $(\Omega',\mathcal{A}',\mathds{P}')=([0,1],\mathcal{B}([0,1]);\lambda)$ according to \cite[Remark after Proposition 1.8.1]{D} we have
\[L^1(\Omega';L^1(0,T;L^1(\Lambda)))\cong L^1(\Omega'\times(0,T);L^1(\Lambda))\cong L^1(0,T;L^1(\Omega';L^1(\Lambda))).\]
For almost every  $t\in (0,T)$ and $0<\delta< t$ we have
\begin{align*}
\left\|u_\infty^\delta(t)-u_\infty(t)\right\|_{L^1(\Omega';L^1(\Lambda))}&=\left\|\frac{1}{\delta}\int_{(t-\delta)^+}^t\left(u_\infty(s)-u_\infty(t)\right)\,ds\right\|_{L^1(\Omega';L^1(\Lambda))}\\
&\leq \frac{1}{\delta}\int_{(t-\delta)^+}^{t}\|u_\infty(s)-u_\infty(t)\|_{L^1(\Omega';L^1(\Lambda))}\,ds.
\end{align*}
By the generalisation of Lebesgue differentiation theorem for vector-valued functions (see, e.g. \cite[Theorem 9, Chapter II]{DU}) the right-hand side of the above inequality goes to $0$ for $\delta\rightarrow 0^{+}$ for almost every $t\in (0,T)$, hence $u_{\infty}^{\delta}(t)\rightarrow u_{\infty}(t)$ a.e in $L^1(\Omega';L^1(\Lambda))$ for $\delta\rightarrow 0^{+}$.
Then, Lebesgue's dominated convergence theorem provides
\begin{align*}
\lim_{\delta\rightarrow 0^+}T_k(u_\infty^\delta)=T_k(u_\infty)
\end{align*}
in $L^1(0,T;L^1(\Omega';L^1(\Lambda)))$, thus also in  $L^1(\Omega'\times(0,T);L^1(\Lambda))$.\\
Thus, passing to a not relabeled subsequence if necessary, $(T_k(u_\infty^\delta(\omega',t)))_{\delta>0}$ 
converges for almost every $(\omega',t)$ in $\Omega'\times(0,T)$ to $T_k(u_\infty(\omega',t))$ in$L^1(\Lambda)$ as $\delta\rightarrow 0^+$ and therefore $T_k(u_\infty(\omega',t))$ has a $d\mathds{P}'\otimes dt$ representative that is $(\mathfrak{F}_t^\infty)_{t\in[0,T]}$-predictable for every $k\in\mathbb{N}$. Obviously there holds
\begin{align*}
u_\infty(\omega', t)=\sup_{k\in\na} T_k(u_\infty(\omega', t))\text{\ in $L^1(\Lambda)$ for a.e. $(\omega', t)$ in } \Omega'\times(0,T),
\end{align*}
where the set of measure zero can be chosen independently of $k\in\na$. This provides the existence of a $d\mathds{P}'\otimes dt$ representative of $u_\infty$ that is $(\mathfrak{F}_t^\infty)_{t\in[0,T]}$-predictable.
\end{proof}

\begin{lem}\label{martingaleIdentification}
For $t\in [0,T]$, $x\in\Lambda$ and $\mathds{P}'$-a.s. in $\Omega'$ we define the stochastic processes
\begin{align*}
\mathcal{M}_{h_m,N_m}(t,x)&:=\int_0^t g(v_{h_m,N_m}^l(s,x))\,dB_{m}(s)\\
\mathcal{M}_{\infty}(t,x)&:=\int_0^t g(u_{\infty}(s,x))\,dW_{\infty}(s).
\end{align*}
Then, passing to a not relabeled subsequence if necessary,
\begin{align}\label{210830_02}
\mathcal{M}_{h_m,N_m}\stackrel{m\rightarrow+\infty}{\longrightarrow} \mathcal{M}_{\infty}
 \text{ in } L^2(0,T;L^2(\Lambda)) \ \mathds{P}'\text{-a.s. in } \Omega'.
\end{align}
\end{lem}

\begin{proof}
From Lemma \ref{210830_lem1}, we know that $(B_m)_m$  converges in $L^2(\Omega';C([0,T]))$ towards $W_\infty$ which is a Brownian motion with respect to $(\mathfrak{F}_t^\infty)_{t\in[0,T]}$. Particularly, this convergence result also holds in probability in $C([0,T])$. Moreover, from the convergence~\eqref{eq:conv} and Lemma~\ref{lem:id_vm}, we know that $(\vhnlm)_m$ converges towards $u_\infty$ in $L^2(0,T;\lzlambda),\ \mathds{P'}\text{-a.s. in }\Omega'$, thus up to a subsequence denoted in the same way, using the Lipschitz property of g, it follows that $(g(\vhnlm))_m$ converges towards $g(u_{\infty})$ in probability in $L^2(0,T;\lzlambda)$. Now, we can apply Lemma 2.1 in \cite{Debussche}  and conclude that the convergence in \eqref{210830_02} holds true in probability in $L^2(0,T;L^2(\Lambda))$ and therefore, passing to a subsequence if necessary, the assertion follows. 
\end{proof}

\subsection{Convergence towards a martingale solution}\label{S3}

For the sake of simplicity we use the notations $\Tau=\Tau_m$, $h=h_m$, $\Delta t= \Delta t_m$ and $N=N_m$.
For any $n \in \{0,\ldots,N\}$ and $K \in \Tau$, setting $\mathcal M_K^n=\mathcal M_{h,N}(t_n,x_K)$ we can define $\widehat{\mathcal{M}}_{h,N}$ using Definition~\eqref{eq:notation_whhat} and we obtain the following strong convergence result in $L^p(\Omega';L^2(0,T;\lzlambda))$.

\begin{lem}\label{convergenceaffine}
Passing to a not relabeled subsequence if necessary, we have the following convergence results for any $p\in[1,2)$:
\begin{align*}
v_{h,N}^l,\ v_{h,N}^r \ \text{and} \ \widehat{v}_{h,N}&\stackrel{m\rightarrow+\infty}{\longrightarrow} u_\infty &&\text{in }L^p(\Omega';L^2(0,T;\lzlambda)),\\
\mathcal{M}_{h,N} \ \text{and} \ \widehat{\mathcal{M}}_{h,N}&\stackrel{m\rightarrow+\infty}{\longrightarrow} \mathcal{M}_\infty &&\text{in }L^p(\Omega';L^2(0,T;\lzlambda))
\end{align*} 
and 
\begin{align*}
v_h^0\stackrel{m\rightarrow+\infty}{\longrightarrow} v_0\quad \text{in }L^p(\Omega';\lzlambda).
\end{align*}
Moreover, $u_{\infty}\in L^2(\Omega';L^2(0,T;H^1(\Lambda)))$ and $\mathcal{M}_{\infty}\in L^2(\Omega';C([0,T];L^2(\Lambda)))$.
\end{lem}

\begin{proof}
We recall that thanks to convergence~\eqref{eq:conv}, $(v_{h,N}^l)_m$ converges to $u_{\infty}$ for $m\rightarrow\infty$ in $L^2(0,T;L^2(\Lambda))$ $\mathds{P}'$-a.s. in $\Omega'$. Since $\mathds{P}'\circ (v_{h,N}^l)^{-1}=\mathds{P}\circ(u_{h,N}^l)^{-1}$, from  Lemma \ref{210611_lem01} it follows that there exists a constant $C\geq 0$ such that
\begin{align}\label{211008_02}
\erws{\| v_{h,N}^l\|_{L^2(0,T;\lzlambda)}^2}\leq C
\end{align}
for all $m\in\mathbb{N}$ and from Fatou's lemma we obtain $u_{\infty}\in L^2(\Omega';L^2(0,T;L^2(\Lambda)))$. The convergence of $(v_{h,N}^l)_m$ towards $u_{\infty}$ in $L^p(\Omega';L^2(0,T;\lzlambda))$ is a consequence of \eqref{211008_02} and of the theorem of Vitali (see, e.g., \cite[Corollaire 1.3.3]{D}). Now, using \eqref{210824_05} we get
\begin{align*}
\erws{\|v_{h,N}^r-v_{h,N}^l\|_{L^2(0,T;\lzlambda)}^2}=\erww{\|\uhnr-\uhnl\|_{L^2(0,T;\lzlambda)}^2}\stackrel{m\rightarrow+\infty}{\longrightarrow}0,
\end{align*}
hence $(v_{h,N}^r-v_{h,N}^l)\rightarrow 0$ in $L^2(\Omega';L^2(0,T;\lzlambda))$ as $m\rightarrow+\infty$ and, thanks to the continuous embedding of $L^2(\Omega';L^2(0,T;\lzlambda))$ into  $L^p(\Omega';L^2(0,T;\lzlambda))$, also in $L^p(\Omega';\linebreak L^2(0,T;\lzlambda))$ for all $1\leq p<2$. Therefore, $v_{h,N}^r\rightarrow u_{\infty}$ as $m\rightarrow+\infty$ in $L^p(\Omega';\linebreak L^2(0,T;\lzlambda))$ for all $1\leq p<2$. 
Now we apply a similar argumentation to $(\widehat{v}_{h,N})_m$. We have
\begin{align}
\label{211008_03}
\begin{aligned}
\erws{\|\vhnl-\widehat{v}_{h,N}\|_{L^2(0,T;\lzlambda)}^2}
&=\erws{\sum_{n=0}^{N-1}\|v_h^{n+1}-v_h^n\|_\lzlambda^2\int_{t_n}^{t_{n+1}}\left(\frac{t-t_n}{\Delta t}\right)^2\, dt}\\
&=\frac{\Delta t}{3}\erws{\sum_{n=0}^{N-1}\|v_h^{n+1}-v_h^n\|_\lzlambda^2}.
\end{aligned}
\end{align}
Repeating the arguments of Proposition~\ref{bounds} on \eqref{newsemiimplicitequ} it follows that there exists a constant $C_1'\geq 0$ such that
\begin{align}\label{210830_07}
\erws{\int_0^T|v_{h,N}^r|_{1,h}^2\,dt}+\erws{\sum_{n=0}^{N-1}\|v_h^{n+1}-v_h^n\|_\lzlambda^2}\leq C_1'.
\end{align}
Combining \eqref{211008_03} with \eqref{210830_07} it follows that $(v_{h,N}^l-\widehat{v}_{h,N})\rightarrow 0$ in $L^2(\Omega';L^2(0,T;\lzlambda))$ and we may conclude that $\widehat{v}_{h,N}\rightarrow u_{\infty}$ in $L^p(\Omega';L^2(0,T;\lzlambda))$ for all $1\leq p<2$.
Using \eqref{210830_07} and the same arguments as in the proof of Lemma \ref{addreg u} on the discrete gradient $\nabla^hv_{h,N}^r$ of $v_{h,N}^r$ it follows that, passing to a not relabeled subsequence if necessary, $(\nabla^h v_{h,N}^r)_m$ converges weakly in $L^2(\Omega';L^2(0,T;L^2(\Lambda)^2))$ towards $\nabla u_{\infty}$, hence $u_{\infty}\in\linebreak L^2(\Omega';L^2(0,T;H^1(\Lambda)))$.\\
We recall that, according to Lemma \ref{martingaleIdentification}, $\mathcal{M}_{h,N}\rightarrow \mathcal{M}_{\infty}$ for $m\rightarrow+\infty$ in $L^2(0,T;\lzlambda)$ $\mathds{P}'$-a.s. in $\Omega'$. Using $H_2$, \eqref{H3}, the Burkholder-Davis-Gundy inequality with constant $C_B\geq 0$ and Lemma \ref{uhnsupbound} it follows that 
\begin{align}\label{210830_06}
\begin{split}
\erws{\sup_{t\in [0,T]}\|\mathcal{M}_{h,N}(t)\|_{\lzlambda}^2}
&\leq C_B \erws{\int_0^T\|g(\vhnl(t))\|_\lzlambda^2\,dt}\\
&\leq C_B C_L\left(|\Lambda| T +\erws{\int_0^T\|\vhnl(t)\|_\lzlambda^2\,dt}\right)\\
&=C_B C_L\left(|\Lambda| T +\erww{\int_0^T\|\uhnl(t)\|_\lzlambda^2\,dt}\right)\\
&= C_BC_LT(|\Lambda|+\cter{K4}).
\end{split}
\end{align}
Now, the convergence of $(\mathcal{M}_{h,N})_m$ towards $\mathcal{M}_{\infty}$ in $L^p(\Omega';L^2(0,T;\lzlambda))$ for all $1\leq p<2$ follows from \eqref{210830_06} and the theorem of Vitali (see, e.g., \cite[Corollaire 1.3.3]{D}).
Using the It\^o isometry, $H_2$, \eqref{H3} and Lemma \ref{uhnsupbound} it follows that there exists $C_3\geq 0$ such that
\begin{align*}
&\erws{\int_0^T\|\mathcal{M}_{h,N}(t)-\widehat{\mathcal{M}}_{h,N}(t)\|_\lzlambda^2\,dt}\\
&=\erws{\sum_{n=0}^{N-1}\int_{t_n}^{t_{n+1}}\left\|\int_{t_n}^tg(\vhnl(s))\,dB_{m}(s)-\frac{t-t_n}{\Delta t}\int_{t_n}^{t_{n+1}}g(\vhnl(s))\,dB_{m}(s)\right\|_\lzlambda^2\,dt}\\
&\leq 2\erws{\sum_{n=0}^{N-1}\int_{t_n}^{t_{n+1}}\left(\int_{t_n}^{t}\| g(\vhnl(s))\|_\lzlambda^2\,ds+\left(\frac{t-t_n}{\Delta t}\right)^2\int_{t_n}^{t_{n+1}}\|g(\vhnl(s))\|_\lzlambda^2\,ds\right)\,dt}\\
&\leq 2\sum_{n=0}^{N-1}\int_{t_n}^{t_{n+1}}C_L(t-t_n)\left(|\Lambda|+\erws{\sup_{s\in [0,T]}\| \vhnl(s)\|_\lzlambda^2}\right)\left(1+\frac{(t-t_n)}{\Delta t}\right)\,dt\\
&\leq 2\sum_{n=0}^{N-1}\int_{t_n}^{t_{n+1}}C_L(t-t_n)\left(|\Lambda|+\erww{\sup_{s\in [0,T]}\| \uhnl(s)\|_\lzlambda^2}\right)\left(1+\frac{(t-t_n)}{\Delta t}\right)\,dt\\
&\leq \frac53C_LT(|\Lambda|+K_3)\Delta t\stackrel{m\rightarrow+\infty}{\longrightarrow 0}.
\end{align*}
Hence, $\mathcal{M}_{h,N}-\widehat{\mathcal{M}}_{h,N}\rightarrow 0$ for $m\rightarrow +\infty$ in $L^2(\Omega';L^2(0,T;\lzlambda))$ and therefore $\widehat{\mathcal{M}}_{h,N}\rightarrow\mathcal{M}_{\infty}$ for  $m\rightarrow +\infty$ in $L^p(\Omega';L^2(0,T;\lzlambda))$ for all $1\leq p<2$.
Recalling that $\mathcal{M}_{\infty}$ is a stochastic It\^{o} integral with respect to the Brownian motion $(W_{\infty}(t))_{t\geq 0}$, we may conclude that $\mathcal{M}_{\infty}$ has $\mathds{P}'$-a.s. continuous paths in $\lzlambda$.
From \eqref{210830_06} and Fatou's lemma it now follows that  $\mathcal{M}_{\infty}\in L^2(\Omega';C([0,T];\lzlambda))$. Since $\mathds{P}'\circ(v_h^0)^{-1}=\mathds{P}\circ(u_h^0)^{-1}$, from Lemma~\ref{bound_u0} it follows that
\begin{align*}
\erws{\| v_h^0\|^2_\lzlambda}=\erww{\| u_h^0\|^2_\lzlambda}\leq \erww{\| u_0\|^2_\lzlambda}.
\end{align*}
Together with the $\mathds{P}'$-a.s. convergence of $v_h^0\stackrel{m\rightarrow+\infty}{\longrightarrow v_0}$ in $\lzlambda$ from the convergence result~\eqref{eq:conv}, the last assertion follows again from the theorem of Vitali (see, e.g., \cite[Corollaire 1.3.3]{D}).
\end{proof}

We now have all the necessary tools to pass to the limit in the scheme.

\begin{prop}\label{210823_p1}
There exists a subsequence of $(\widehat{v}_{h,N})_m$, still denoted by $({\widehat{v}_{h,N}})_m$, converging in $L^p(\Omega';L^2(0,T;\lzlambda))$ (for any $p\in[1,2)$) as $m\rightarrow+\infty$ towards a $(\mathfrak{F}_t^\infty)_{t\in[0,T]}$-adapted stochastic process $u_{\infty}$ with values in $L^2(\Lambda)$ and having $\mathds{P}'$-a.s. continuous paths. Moreover, $u_{\infty}\in L^2(\Omega';L^2(0,T;H^1(\Lambda)))$ and satisfies for all $t\in [0,T]$,
\[u_{\infty}(t)-v_0-\int_0^t\Delta u_{\infty}\,ds=\int_0^t g(u_{\infty})\,dW_{\infty}
\quad \text{in $L^2(\Lambda)$ and $\mathds{P}'$-a.s. in $\Omega'$.}\]
\end{prop}

\begin{proof}
Let $A\in\mathcal{A}'$, $\xi\in \mathcal{D}(\re)$ with $\xi(T)=0$ and $\varphi\in \mathcal{D}(\re^2)$ with $\nabla\varphi\cdot\mathbf{n}=0$ on $\partial\Lambda$, where we denote $\mathcal{D}(D):=C_c^\infty(D)$ for any open subset $D\subseteq \re^j,j\in\na$. Moreover we define the piecewise constant function $\varphi_h(x):=\varphi(x_K)$ for $x\in K$, $K\in\Tau$.\\
For $K\in\Tau$, $n\in\{0,\dots,N-1\}$ and $t\in[t_n,t_{n+1})$ we multiply \eqref{newsemiimplicitequ} with $\mathds{1}_A\xi(t)\varphi(x_K)$ to obtain
\begin{equation}\label{220814}
\mathds{1}_A\xi(t)\frac{m_K}{\Delta t}[v_K^{n+1}-v_K^n-g(v_K^n)\Delta_{n+1}B_m]\varphi(x_K)+\mathds{1}_A\xi(t)\sum_{\sigma\in\edgesint\cap\edges_K}\frac{m_\sigma}{\dkl}(v_K^{n+1}-v_L^{n+1})\varphi(x_K)=0.
\end{equation}
First we sum \eqref{220814} over each control volume $K\in\Tau$, then we integrate over each time interval $[t_{n},t_{n+1}]$ for fixed $n=0,\dots,N-1$, then we sum over $n=0,\dots,N-1$ and finally we take the expectation to obtain
\begin{equation}
\label{discretequforlimit}
\begin{aligned}
0&=\mathbb{E}'\left[\sum_{n=0}^{N-1}\int_{t_n}^{t_{n+1}}\sum_{K\in\Tau}m_K\mathds{1}_A\xi(t)\frac{1}{\Delta t}[v_K^{n+1}-v_K^n-g(v_K^n)\Delta_{n+1}B_m]\varphi(x_K)\,dt\right]\\
&\quad+\mathbb{E}'\left[\sum_{n=0}^{N-1}\int_{t_n}^{t_{n+1}}\mathds{1}_A\xi(t)\sum_{K\in\Tau}\sum_{\sigma\in\edgesint\cap\edges_K}\frac{m_\sigma}{\dkl}(v_K^{n+1}-v_L^{n+1})\varphi(x_K)\,dt\right]\\
&=:T_{1,m}+T_{2,m}.
\end{aligned}
\end{equation}
In the following, we will pass to the limit with $m\rightarrow+\infty$ on the right-hand side of \eqref{discretequforlimit}.
Using partial integration we obtain
\begin{align*}
T_{1,m}
&=\mathbb{E}'\left[\mathds{1}_A\int_0^T\int_\Lambda\partial_t[\widehat{v}_{h,N}-\widehat{\mathcal{M}}_{h,N}](t,x)\xi(t)\varphi_h(x)\,dx\,dt\right]\\
&=-\mathbb{E}'\left[\mathds{1}_A\int_0^T\int_\Lambda[\widehat{v}_{h,N}-\widehat{\mathcal{M}}_{h,N}](t,x)\xi'(t)\varphi_h(x)\,dx\,dt\right]-\mathbb{E}'\left[\mathds{1}_A\int_\Lambda v_h^0(x)\xi(0)\varphi_h(x)\,dx\right].
\end{align*}
Thanks to the convergence results of Lemma \ref{convergenceaffine}, passing to a not relabeled subsequence if necessary, we can pass to the limit and obtain
\begin{align*}
&-\mathbb{E}'\left[\mathds{1}_A\int_0^T\int_\Lambda[\widehat{v}_{h,N}-\widehat{\mathcal{M}}_{h,N}](t,x)\xi'(t)\varphi_h(x)\,dx\,dt\,\right]-\mathbb{E}'\left[\mathds{1}_A\int_\Lambda v_h^0(x)\xi(0)\varphi_h(x)\,dx\right]\\
&\stackrel{m\rightarrow+\infty}{\longrightarrow} -\mathbb{E}'\left[\mathds{1}_A\int_0^T\int_\Lambda[u_\infty-\mathcal{M}_\infty](t,x)\xi'(t)\varphi(x)\,dx\,dt\right]-\mathbb{E}'\left[\mathds{1}_A\int_\Lambda v_0(x)\xi(0)\varphi(x)\,dx\right].
\end{align*}
Now our aim is to show the following convergence result:
\begin{align*}
T_{2,m}\stackrel{m\rightarrow+\infty}{\longrightarrow} -\mathbb{E}'\left[\mathds{1}_A\int_0^T\int_\Lambda\xi(t)\Delta\varphi(x)u_\infty(t,x)\,dx\,dt\right].
\end{align*}
First, we note that by rearranging the sum in \eqref{discretequforlimit} the term $T_{2,m}$ can be written as
\[
 T_{2,m}
 =\mathbb{E}'\left[\sum_{n=0}^{N-1}\int_{t_n}^{t_{n+1}}\mathds{1}_A\xi(t)
 \sum_{K\in\Tau}v_K^{n+1}
 \sum_{\sigma\in\edgesint\cap\edges_K}m_\sigma\left(\frac{\varphi(x_K)-\varphi(x_L)}{\dkl}\right)\,dt\right].
\]
Then, since $\nabla\varphi\cdot\mathbf{n}=0$ on $\partial\Lambda$, thanks to the Stokes formula one has,
\[
 \int_K\Delta\varphi(x)\,dx
 = \int_{\partial K}\nabla\varphi(x)\cdot\mathbf{n}\,d\sigma(x)
 = \sum_{\sigma\in\edgesint\cap\edges_K} \int_\sigma \nabla\varphi(x)\cdot\mathbf{n}_{KL}\,d\sigma(x).
\]
Thus, we have
\begin{align*}
 &T_{2,m}=\\
 &-\mathbb{E}'\left[\sum_{n=0}^{N-1}\int_{t_n}^{t_{n+1}}\mathds{1}_A\xi(t)
 \sum_{K\in\Tau} v_K^{n+1}
 \left(\int_K\Delta\varphi(x)\,dx - \sum_{\sigma\in\edgesint\cap\edges_K} \int_\sigma \nabla\varphi(x)\cdot\mathbf{n}_{KL}\,d\sigma(x) \right)\,dt\right] \\
&+ \mathbb{E}'\left[\sum_{n=0}^{N-1}\int_{t_n}^{t_{n+1}}\mathds{1}_A\xi(t)
 \sum_{K\in\Tau} v_K^{n+1}\sum_{\sigma\in\edgesint\cap\edges_K}m_\sigma\left(\frac{\varphi(x_K)-\varphi(x_L)}{\dkl}\right)\,dt\right]\\
 &= -\mathbb{E}'\left[\int_0^T\mathds{1}_A\xi(t)\int_\Lambda \vhnr(t,x)\Delta\varphi(x)\,dx\,dt\right]\\
&+ \mathbb{E}'\left[\sum_{n=0}^{N-1}\int_{t_n}^{t_{n+1}}\mathds{1}_A\xi(t)
\sum_{\sigma\in\edgesint} m_\sigma (v_K^{n+1}-v_L^{n+1}) R_\sigma^\varphi
 \,dt\right],
\end{align*}
with
\[
 R_\sigma^\varphi
 = \frac1{m_\sigma} \int_\sigma \nabla\varphi(x)\cdot\mathbf{n}_{KL}\,d\sigma(x)
 - \frac{\varphi(x_L)-\varphi(x_K)}{\dkl}.
\]
Using Lemma \ref{convergenceaffine} and passing to a not relabeled subsequence if necessary, one gets
\begin{align*}
-\mathbb{E}'\left[\int_0^T\int_\Lambda\mathds{1}_A\xi(t)v_{h,N}^r(t,x)\Delta\varphi(x)\,dx\,dt\right]
\stackrel{m\rightarrow+\infty}{\longrightarrow}-\mathbb{E}'\left[\int_0^T\int_\Lambda\mathds{1}_A\xi(t)u_\infty(t,x)\Delta\varphi(x)\,dx\,dt\right].
\end{align*}
Concerning the second term in $T_{2,m}$, for any $\sigma=K|L\in\edgesint$, the orthogonality condition implies $x_L-x_K=\dkl\mathbf{n}_{KL}$, thus thanks to the Taylor formula for any $x\in\sigma$ one has,
\[
 \nabla\varphi(x)\cdot\mathbf{n}_{KL}=\frac{\varphi(x_L)-\varphi(x_K)}{\dkl}+\mathcal O(h),
\]
that gives,
\[
 R_\sigma^\varphi \leq C_\varphi h.
\]
Therefore, thanks to the Cauchy-Schwarz inequality and inequality~\eqref{210830_07} the second term in $T_{2,m}$ satisfies,
\begin{align*}
 &\left|\mathbb{E}'\left[\sum_{n=0}^{N-1}\int_{t_n}^{t_{n+1}}\mathds{1}_A\xi(t)
\sum_{\sigma\in\edgesint} m_\sigma (v_K^{n+1}-v_L^{n+1}) R_\sigma^\varphi
 \,dt\right]\right|\\
 &\leq C_\varphi h\mathbb{E}'\left[\sum_{n=0}^{N-1}\int_{t_n}^{t_{n+1}}\mathds{1}_A |\xi(t)|
 \left(\sum_{\sigma\in\edgesint}m_\sigma\dkl\right)^\frac12 \left(\sum_{\sigma\in\edgesint}m_{\sigma}\frac{|v_K^{n+1}-v_L^{n+1}|^2}{\dkl}\right)^\frac12\right]\\
 &\leq \sqrt2 C_\varphi |\Lambda|^\frac12 h\mathbb{E}'\left[\int_0^T\mathds{1}_A |\xi(t)||\vhnr(t)|_{1,h}\right]\\
 &\leq \sqrt2 C_\varphi |\Lambda|^\frac12 h\|\xi\mathds{1}_A\|_{L^2(\Omega'\times(0,T))}\left(\erws{\int_0^T|v_{h,N}^r(t)|_{1,h}^2\,dt}\right)^\halbe
\stackrel{m\rightarrow+\infty}{\longrightarrow} 0. 
\end{align*}
Thus, we have shown that 
\begin{align}\label{210830_08}
\begin{split}
&-\int_0^T\int_\Lambda\left(u_\infty(t,x)-\int_0^{t}g(u_\infty(s,x))\,dW_\infty(s)\right)\xi'(t)\varphi(x)\,dx\,dt-\int_\Lambda v_0(x)\xi(0)\varphi(x)\,dx\\
&=\int_0^T\int_\Lambda  u_\infty(t,x)\Delta\varphi(x)\xi(t)\,dx\,dt =-\int_0^T \int_\Lambda \nabla u_\infty(t,x) \cdot \nabla \varphi(x)\xi(t)\,dx\,dt
\end{split}
\end{align}
$\mathds{P}'$-a.s. in $\Omega'$ for all $\xi\in \mathcal{D}(\re)$ with $\xi(T)=0$ and all $\varphi\in \mathcal{D}(\re^2)$ such that $\nabla\varphi\cdot\mathbf{n}=0$ on $\partial\Lambda$. By \cite[Theorem 1.1]{Droniou} the set $\{\varphi\in\mathcal{D}(\mathbb{R}^2)\ | \ \nabla\varphi\cdot\mathbf{n}=0 \ \text{on} \ \partial\Lambda\}$ is dense in $H^1(\Lambda)$ and therefore \eqref{210830_08} applies to all $\varphi\in H^1(\Lambda)$. 

In the following, we denote the dual space of $H^1(\Lambda)$ by $H^1(\Lambda)^{\ast}$, recall that
\[H^1(\Lambda)\hookrightarrow L^2(\Lambda)\hookrightarrow H^1(\Lambda)^{\ast}\]
with continuous and dense embeddings and we will denote the $H^1(\Lambda)$-$H^1(\Lambda)^{\ast}$ duality bracket by $\langle\cdot,\cdot\rangle$ and the $L^2(\Lambda)$ scalar product by $(\cdot,\cdot)$. With the additional information 
\[u_{\infty}\in L^2(\Omega';L^2(0,T;H^1(\Lambda)))\]
from Lemma \ref{convergenceaffine}, it follows that 
\[\Delta u_{\infty}\in L^2(\Omega';L^2(0,T;H^1(\Lambda)^{\ast}))\]
and
\begin{align}\label{220125_01}
\begin{split}
-\int_0^T\int_{\Lambda}\nabla u_{\infty}(t,x)\cdot\nabla\varphi(x)\xi(t)\,dx\,dt=\int_0^T\langle \Delta u_{\infty}(t,\cdot),\varphi\rangle\xi(t)\,dt
\end{split}
\end{align}
$\mathds{P}'$-a.s. in $\Omega'$, for all $\xi\in \mathcal{D}(\mathbb{R})$ such that $\xi(T)=0$ and all $\varphi\in H^1(\Lambda)$. Combining \eqref{210830_08} with \eqref{220125_01} and with the identity
\begin{align}\label{eqliminitialdata}
-\int_{\Lambda}v_0(x)\varphi(x)\xi(0)\,dx=\int_0^T\int_{\Lambda}v_0(x)\varphi(x)\xi'(t)\,dx\,dt
\end{align}
(see, \cite[Lemma 7.3]{Roubicek}), from Fubini's theorem it follows that
\begin{align*}
\left\langle-\int_0^T\left(u_{\infty}(t)-\int_0^t g(u_{\infty})\,dW_{\infty}-v_0\right)\xi'(t)\,dt,\varphi\right\rangle=\left\langle\int_0^T\Delta u_{\infty}(t)\xi(t)\,dt,\varphi\right\rangle
\end{align*}
$\mathds{P}'$-a.s. in $\Omega'$ for all $\xi\in \mathcal{D}(\mathbb{R})$ such that $\xi(T)=0$ and all $\varphi\in H^1(\Lambda)$. Therefore
\begin{align*}
-\int_0^T\left(u_{\infty}(t)-\int_0^t g(u_{\infty})\,dW_{\infty}-v_0\right)\xi'(t)\,dt=\int_0^T\Delta u_{\infty}(t)\xi(t)\,dt
\end{align*}
in $H^1(\Lambda)^{\ast}$, for all $\xi\in \mathcal{D}(\mathbb{R})$ such that 
$\xi(T)=0$, $\mathds{P}'$-a.s. in $\Omega'$ since, by a separablity argument, the exceptional set in $\Omega'$ may be chosen independently of $\varphi$. Consequently, (see, e.g. \cite[Proposition A6]{Brezis})
\[u_{\infty}-\int_0^\cdot g(u_{\infty})\,dW_{\infty}-v_0\in W^{1,2}(0,T;H^1(\Lambda)^{\ast})\quad \text{$\mathds{P}'$-a.s. in $\Omega'$}\]
and
\begin{align}\label{220126_02}
\frac{d}{dt}\left(u_{\infty}(t)-\int_0^t g(u_{\infty})\,dW_{\infty}-v_0\right)=\Delta u_{\infty}\quad \text{in} \  L^2(\Omega';L^2(0,T;H^1(\Lambda)^{\ast}).
\end{align}
Since $g$ is Lipschitz continuous, from the chain rule for Sobolev functions it follows that $g(u_{\infty})\in L^2(\Omega';L^2(0,T;H^1(\Lambda)))$ and
\[\nabla\left(\int_0^t g(u_{\infty})\,dW_{\infty}\right)=\int_0^tg'(u_{\infty})\nabla u_{\infty}\,dW_{\infty},\]
hence $u_{\infty}-\int_0^{\cdot}g(u_{\infty})\,dW_{\infty}\in L^2(\Omega';L^2(0,T;H^1(\Lambda)))$. From  \cite[Lemma 7.3]{Roubicek} we obtain $u_{\infty}\in L^2(\Omega';C([0,T];L^2(\Lambda))$ and together with \eqref{220126_02} the following rule of partial integration for all $0\leq t\leq T$, $\mathds{P}'$-a.s. in $\Omega'$:  
\begin{align}\label{220126_01}
\begin{split}
&\left(u_{\infty}(t)-\int_0^t g(u_{\infty})\,dW_{\infty}-v_0,\zeta(t)\right)-(u_{\infty}(0)-v_0,\zeta(0))\\
&=\int_0^t\left\langle\Delta u_{\infty}(s),\zeta(s)\right\rangle\,ds+\int_0^t\left\langle \zeta'(s),u_{\infty}(s)-\int_0^s g(u_{\infty})\,dW_{\infty}-v_0\right\rangle\,ds
\end{split}
\end{align}
for all $\zeta\in L^2(0,T;H^1(\Lambda))$ with $\zeta'\in L^2(0,T;H^1(\Lambda)^{\ast}))$. Choosing $\zeta(t,x)=\xi(t)\varphi(x)$ with $\varphi\in H^1(\Lambda)$, $\xi\in\mathcal{D}(\mathbb{R})$ with $\xi(T)=0$ in \eqref{220126_01}, we get
\begin{align}\label{220126_03}
\begin{split}
&\left(u_{\infty}(t)-\int_0^t g(u_{\infty})\,dW_{\infty}-v_0,\varphi\right)\xi(t)-(u_{\infty}(0)-v_0,\varphi)\xi(0)\\
&=\int_0^t\xi(s)\left\langle\Delta u_{\infty}(s),\varphi\right\rangle\,ds+\int_0^t \xi'(s)\left(u_{\infty}(s)-\int_0^s g(u_{\infty})\,dW_{\infty}-v_0,\varphi\right)\,ds
\end{split}
\end{align}
$\mathds{P}'$-a.s. in $\Omega'$. The particular choice of $t=T$ and $\xi\in\mathcal{D}(\mathbb{R})$ with $\xi(T)=0$ and $\xi(0)=1$ in \eqref{220126_03} combined with \eqref{210830_08}, \eqref{220125_01} and \eqref{eqliminitialdata} yields
\[(u_{\infty}(0)-v_0,\varphi)=0 \quad \text{for all $\varphi\in H^1(\Lambda)$, $\mathds{P}'$-a.s. in $\Omega'$}\]
and therefore $u_{\infty}(0)=v_0$ $\mathds{P}'$-a.s. in $\Omega'$.\\
Now, we fix $t\in [0,T)$ and choose $\xi\in\mathcal{D}(\mathbb{R})$ with $\xi(T)=0$ and $\xi(s)=1$ for all $s\in [0,t]$. With this choice, from \eqref{220126_03} we obtain
\begin{align}\label{220127_01}
\left(u_{\infty}(t)-\int_0^t g(u_{\infty})\,dW_{\infty}-u_{\infty}(0),\varphi\right)=\int_0^t\left\langle\Delta u_{\infty}(s),\varphi\right\rangle\,ds
\end{align}
$\mathds{P}'$-a.s. in $\Omega'$ for all $\varphi\in H^1(\Lambda)$. Since, for fixed $\varphi\in H^1(\Lambda)$, 
\[t\mapsto \left(u_{\infty}(t)-\int_0^t g(u_{\infty})\,dW_{\infty}-u_{\infty}(0),\varphi\right) \quad \text{and} \quad t\mapsto \int_0^t\left\langle\Delta u_{\infty}(s),\varphi\right\rangle\,ds\]
are continuous in $[0,T]$, $\mathds{P}'$-a.s. in $\Omega'$, the exceptional set in $\Omega'$ in \eqref{220127_01} may be chosen independently of $t\in [0,T)$ and \eqref{220127_01} also holds for $t=T$. This yields
\begin{align*}
u_{\infty}(t)-u_{\infty}(0)-\int_0^t g(u_{\infty})\,dW_{\infty}=\int_0^t\Delta u_{\infty}(s)\,ds \quad \text{in $H^1(\Lambda)^{\ast}$ and $\mathds{P}'$-a.s. in $\Omega'$}
\end{align*}
and, since the left-hand side of the above equation is in $L^2(\Lambda)$, the equation holds also in $L^2(\Lambda)$.
\end{proof}

\begin{remark}
Applying the chain rule in \eqref{220126_01} for $t=T$ and $\zeta=\Psi\in \mathcal{D}(\mathbb{R}\times\mathbb{R}^2)$ such that $\Psi(T,\cdot)=0$ we immediately get that $u_{\infty}$ is a weak solution, i.e.,
\begin{align*}
&\int_0^T\int_{\Lambda}u_{\infty}(t,x)\partial_t\Psi(t,x)\,dx\,dt-\int_0^T\int_{\Lambda}\nabla u_{\infty}(t,x)\cdot\nabla\Psi(t,x)\, dx\,dt+\int_{\Lambda}u_0(x)\Psi(0,x)\,dx\\
&=\int_0^T\int_{\Lambda}\int_0^tg(u_{\infty}(s,x))\,dW_{\infty}(s)\partial_t\Psi(t,x)\,dx\,dt
\end{align*}
$\mathds{P}'$-a.s. in $\Omega'$. In particular convergence in distribution has been achieved.
\end{remark}

\subsection{Strong convergence of finite-volume approximations}
In the previous subsections, we have shown that our finite-volume approximations converge towards a \textit{martingale solution} of \eqref{equation}, i.e., the stochastic basis 
\[(\Omega',\mathcal{A}',\mathds{P}',(\mathfrak{F}^{\infty}_t)_{t\in[0,T]}, (W_{\infty}(t))_{t\in[0,T]})\]
is not a-priori given, but part of the solution. In this subsection, we want to show convergence of our finite-volume approximations with respect to the initially given stochastic basis 
\[(\Omega,\mathcal{A},\mathds{P}, (\mathcal{F}_t)_{t\geq 0}, (W(t))_{t\geq 0}).\]
To do so, we will proceed in several steps. First, \textit{pathwise uniqueness} of the heat equation with multiplicative Lipschitz noise is a consequence of Proposition \ref{210902_p1}: Roughly speaking, martingale solutions of \eqref{equation} on a joint stochastic basis and with respect to the same initial datum coincide. In the proof of Proposition \ref{finalProp}, we construct two convergent finite-volume approximations with respect to a joint stochastic basis, namely $(v_{\nu_k}^l)$ and $(v_{\rho_k}^l)$, from the function $(\uhnl)$ of our original finite-volume scheme using the theorems of Prokhorov and Skorokhod. Then, as a consequence of pathwise uniqueness, the limits coincide and we may apply \cite[Lemma 1.1]{Krylov} in order to obtain the convergence in probability of $(\uhnl)$. Thanks to our previous result we can improve the convergence and pass to the limit in the originally given finite-volume scheme (see Lemma \ref{220315}).
\begin{prop}\label{210902_p1} 
Let $(\Omega,\mathcal{A},\mathds{P},(\mathcal{F}_t)_{t\geq 0},(W(t))_{t\geq 0})$ be a stochastic basis and $u_1$, $u_2$ be solutions to \eqref{equation} with respect to the $\mathcal{F}_0$-measurable initial values $u_0^1$ and $u_0^2$ in $L^2(\Omega;\lzlambda)$ respectively on $(\Omega,\mathcal{A},\mathds{P},(\mathcal{F}_t)_{t\geq 0},(W(t))_{t\geq 0})$. Then, there exists a constant $C\geq 0$ such that
\begin{align*}
\erww{\| u_1(t)-u_2(t)\|^2_{\lzlambda}}\leq C\erww{\|u_0^1-u_0^2\|^2_{\lzlambda}}
\end{align*}
for all $t\in [0,T]$.
\end{prop}
\begin{proof}
We apply the It\^{o} formula (\cite[Theorem 4.2.5]{LR}) to the process $u_1-u_2$, discard the nonnegative term on the left-hand side of the resulting equation and take expectation. Then, the assertion is a straightforward consequence of Gronwall's inequality, see \cite[Proposition 2.4.10]{LR}.
\end{proof}
\begin{remark}\label{210902_r1}
If $u_1$, $u_2$ are solutions to \eqref{equation} on $(\Omega,\mathcal{A},\mathds{P},(\mathcal{F}_t)_{t\geq 0},(W(t))_{t\geq 0})$ with respect to the same initial value $u_0$, from Proposition \ref{210902_p1} it follows that $u_1(t)=u_2(t)$ in $L^2(\Lambda)$ for all $t\in [0,T]$, $\mathds{P}$-a.s. in $\Omega$. Since $u_1$, $u_2$ have continuous paths in $L^2(\Lambda)$, the exceptional set in $\Omega$ may be chosen independently of $t\in [0,T]$ and it follows that $u_1=u_2$ in $L^2(0,T;L^2(\Lambda))$ $\mathds{P}$-a.s. in $\Omega$.
\end{remark}

\begin{prop}\label{finalProp}
Let $({u}_{h,N}^l)_m$ be given by Proposition \ref{210609_prop1}. Then, there exists a subsequence of $(u^l_{h,N})_m$, still denoted by $(u^l_{h,N})_m$, converging for $m\rightarrow+\infty$ in $L^p(\Omega;L^2(0,T;L^2(\Lambda)))$ for any $p\in[1,2)$ towards the stochastic process $u$ with values in $L^2(\Lambda)$ introduced in Lemma \ref{addreg u}. Moreover, $u$ has $\mathds{P}$-a.s. continuous paths and belongs to $ L^2(\Omega;L^2(0,T;H^1(\Lambda))$.
\end{prop} 

\begin{proof}
For the sake of simplicity, we will write $u^l_m:=\uhnlm$ and $u^r_m:=\uhnrm$ in the following. We consider an arbitrary pair of subsequences $(u^l_{\nu})_{\nu}$, $(u^l_{\rho})_{\rho}$ of $(u^l_m)_m$.
Our aim is to apply \cite[Lemma 1.1]{Krylov}, therefore we show that there exists a joint subsequence $(u^l_{\nu_k},u^l_{\rho_k})_k$ converging in law to a probability measure $\eta$ on $L^2(0,T;\lzlambda)^2$ such that 
\begin{align*}
\eta(\{(x,y)\in L^2(0,T;\lzlambda)^2 \ | \ x=y\})=1.
\end{align*}
We define the random vector-valued sequence $(Y_{\nu,\rho})_{\nu,\rho}$ by 
\[Y_{\nu,\rho}:=(u_{\nu}^l,u_{\rho}^l, (u_{\nu}^r-u_{\nu}^l),(u_{\rho}^r-u_{\rho}^l), W, u_{\nu}^0,u_{\rho}^0,(u_{\nu}^0-u_{\rho}^0))\]
for any $\nu,\rho\in\na^\ast$ and extract a joint subsequence
\[Y_k:=(u_{\nu_k}^l,u_{\rho_k}^l, (u_{\nu_k}^r-u_{\nu_k}^l),(u_{\rho_k}^r-u_{\rho_k}^l), W, u_{\nu_k}^0,u_{\rho_k}^0,(u_{\nu_k}^0-u_{\rho_k}^0))\]
for any $k\in\na$ that converges in law towards a probability measure $\eta_{\infty}$ with marginals $\eta_{\infty}^1$,$\eta_{\infty}^2$, $\delta_0$, $\delta_0$, $\mathds{P}\circ W^{-1}$, $\mathds{P}\circ (u_0)^{-1}$, $\mathds{P}\circ (u_0)^{-1}$, $\delta_0$,
where we include the difference of the random initial data into the vector to ensure that $u_{\nu^k}^0$ and $u_{\rho^k}^0$ converge to the same limit.
With straightforward modifications of the arguments of Subsections \ref{S1}-\ref{S3}, we can find 
random elements $u_{\infty}^1,u_{\infty}^2,v_0$,
\[Y'_k=(v_{\nu_k}^l,v_{\rho_k}^l, z_{\nu_k},z_{\rho_k}, W_k, v_{\nu_k}^0,v_{\rho_k}^0,(v_{\nu_k}^0-v_{\rho_k}^0))\]
such that
\begin{align*}
&\mathds{P}'\circ(Y_k')^{-1}=\mathds{P}\circ(Y_k)^{-1} \ \text{for all} \ k\in\na,\\
&\mathds{P}'\circ(v_0)^{-1}=\mathds{P}\circ(u_0)^{-1},\\
&\mathds{P}'\circ (u^1_{\infty})^{-1}=\eta^1_{\infty},\\
&\mathds{P}'\circ (u^2_{\infty})^{-1}=\eta^2_{\infty}
\end{align*}
and a joint stochastic basis $(\Omega',\mathcal{A}',\mathds{P}', (\mathfrak{F}^{\infty}_t)_{t\in [0,T]}, (W_{\infty}(t))_{t\in [0,T]})$ such that $u^1_{\infty}$ and $u^2_{\infty}$ are both solutions of \eqref{equation} with initial value $v_0$ on $(\Omega',\mathcal{A}',\mathds{P}', (\mathfrak{F}^{\infty}_t)_{t\in [0,T]}, (W_{\infty}(t))_{t\in [0,T]})$. Thus, from Proposition \ref{210902_p1} and Remark \ref{210902_r1} it follows for $\eta=(\eta_{\infty}^1,\eta_{\infty}^2)$
\begin{align*}
1=\mathds{P}'(\{u_{\infty}^1=u_{\infty}^2\})
&=\mathds{P}'\circ(u_{\infty}^1,u_{\infty}^2)^{-1}(\{(x,y)\in L^2(0,T;\lzlambda)^2 \ | \ x=y\})\\
&=\eta(\{(x,y)\in L^2(0,T;\lzlambda)^2 \ | \ x=y\}).
\end{align*}
Then from \cite{Krylov} (Lemma 1.1) we get the convergence of $(u^l_m)_m$ in probability to a random element $\tilde{u}$ in $L^2(0,T, L^2(\Lambda))$. Obviously by Lemma 4.1, $u=\tilde{u}$. This convergence in probability allows us to extract a not relabeled subsequence of $(u^l_m)_m$ that converges $\mathbb{P}$ a.s in $L^2(0,T;L^2(\Lambda))$ towards $u$. Combining this with the boundedness results on $(u^l_m)_m$ in $L^2(\Omega;L^2(0,T;L^2(\Lambda)))$ given by Lemma \ref{210611_lem01}, the application of Vitali's theorem leads us to the announced strong convergence result in $L^p(\Omega;L^2(0,T;L^2(\Lambda)))$, for any $1\leq p <2$.
\end{proof}

At last, to conclude the proof of Theorem \ref{mainresult} which states the convergence of the finite-volume scheme defined by \eqref{eq:def_u0}-\eqref{equationapprox}  towards the variational solution  of the multiplicative stochastic heat equation studied in this paper, it remains to show that the obtained limit $u$ is solution of Problem \eqref{equation} in the sense of Definition \ref{solution}. This is the aim of the following last lemma:

\begin{lem}\label{220315}
The stochastic process $u$ introduced in Lemma \ref{addreg u} is the unique solution of Problem \eqref{equation} in the sense of Definition \ref{solution}. 
\end{lem}

\begin{proof}
Let $p\in[1,2)$. With similar arguments as in the proof of in Proposition \ref{finalProp} it follows that $u_{h,N}^r$ and $\widehat{u}_{h,N}$ converge to $u$ in
$L^p(\Omega;L^2(0,T;L^2(\Lambda)))$
for $m\rightarrow\infty$. Moreover there holds $g(u_{h,N}^l)\overset{m\rightarrow\infty}{\longrightarrow}g(u)$ in
$L^p(\Omega;L^2(0,T;L^2(\Lambda)))$.
Therefore
\begin{align*}
M_{h,N}=\int_0^\cdot g(\uhnl)\,dW\overset{m\rightarrow\infty}\longrightarrow\int_0^\cdot g(u)\,dW\quad\text{in }L^p(\Omega;C([0,T];L^2(\Lambda))).
\end{align*}
As shown in Lemma \ref{convergenceaffine} it follows $\widehat{M}_{h,N}\overset{m\rightarrow\infty}\longrightarrow\int_0^\cdot g(u)\,dW$ in $L^p(\Omega;L^2(0,T;L^2(\Lambda)))$. Now we consider the semi-implicit finite-volume scheme \eqref{equationapprox}, multiply it with $\mathds{1}_A\xi\varphi$, where $A\in\mathcal{A}$, $\xi\in\mathcal{D}(\re)$ with $\xi(T)=0$ and $\varphi\in\mathcal{D}(\re^2)$ with $\nabla \varphi\cdot\mathbf{n}=0$, then we sum over $K\in\Tau$, integrate over $[t_n,t_{n+1})$ and sum over $n=0,\dots,N-1$ to get
\begin{align*}
0=&\mathbb{E}\left[\sum_{n=0}^{N-1}\int_{t_n}^{t_{n+1}}\sum_{K\in\Tau}m_K\mathds{1}_A\xi(t)\frac{1}{\Delta t}[u_K^{n+1}-u_K^n-g(u_K^n)\Delta_{n+1} W]\varphi(x_K)\,dt\right]\\
&+\mathbb{E}\left[\sum_{n=0}^{N-1}\int_{t_n}^{t_{n+1}}\sum_{K\in\Tau}\mathds{1}_A\xi(t)\sum_{K\in\Tau}\sum_{\sigma\in\edgesint\cap\edges_K}\frac{m_\sigma}{\dkl}(u_K^{n+1}-u_L^{n+1})\varphi(x_K)\,dt\right]\\
=:&T_{1,m}+T_{2,m}.
\end{align*}
If we define $\varphi_h(x):=\varphi(x_K)$ for $x\in K$, $K\in\Tau$, there holds
\begin{align*}
T_{1,m}&=\mathbb{E}\left[\mathds{1}_A\int_0^T\int_\Lambda\partial_t[\widehat{u}_{h,N}-\widehat{M}_{h,N}](t,x)\xi(t)\varphi_h(x)\,dx\,dt\right]\\
&=-\mathbb{E}\left[\mathds{1}_A\int_0^T\int_\Lambda(\widehat{u}_{h,N}-\widehat{M}_{h,N})(t,x)\xi'(t)\varphi_h(x)\,dx\,dt\right]-\mathbb{E}\left[\mathds{1}_A\int_\Lambda u_h^0(x)\xi(0)\varphi_h(x)\,dx\right].
\end{align*}
From \cite[Proposition 3.5]{Andreianov} we know that $u_h^0\overset{m\rightarrow+\infty}\longrightarrow u_0$ in $L^2(\Lambda)$, $\mathds{P}$-a.s. in $\Omega$, and thanks to Lemma~\ref{bound_u0} we can apply Lebesgue's dominated convergence theorem. The passage to the limit is analogous to that on $\Omega'$.
\end{proof}
\quad\\
\textbf{Acknowledgments} This work has been supported by the SIMALIN project (ANR-19-CE40-0016) of the French National Research Agency, by the German Research Foundation project ZI 1542/3-1 and the Procope programme for Project-Related Personal Exchange (France-Germany).

\nocite{*}

\bibliographystyle{plain}
\bibliography{FV_Bauzet_Nabet_Schmitz_Zimmermann_220923.bib}

\end{document}